\documentclass[reqno,a4paper,12pt]{amsart} 

\usepackage{amsmath,amscd,amsfonts,amssymb}
\usepackage{mathrsfs,dsfont}

\usepackage{tikz}
\usetikzlibrary{decorations.pathreplacing}
\usetikzlibrary{intersections}

\definecolor{myblue}{rgb}{0.2,0.2,0.5}
\definecolor{myred}{rgb}{0.9,0.2,0.2}
\definecolor{mypurple}{rgb}{0.7,0.15,0.7}
\definecolor{mygreen}{rgb}{0.25,0.85,0.45}
\definecolor{myyellow}{rgb}{0.95,0.65,0.3}
\tikzstyle{mystyle}=[line width=0.3mm, myblue] 
\tikzstyle{mydasheda}=[dash pattern=on 6pt off 7pt]
\tikzstyle{mydashedb}=[dash pattern=on 5pt off 4pt]
\def\mycirc{1.5cm}

\numberwithin{equation}{section}
\numberwithin{figure}{section}

\newcommand\R{\mathbb{R}}

\newcommand\Z{\mathbb{Z}}

\newcommand\gam{\gamma}

\newcommand\lam{\lambda}
\newcommand\Lam{\Lambda}
\newcommand\Sig{\Sigma}
\newcommand\Om{\Omega}
\newcommand\1{\mathds{1}}
\newcommand\eps{\varepsilon}

\renewcommand\ge{\geqslant}
\renewcommand\leq{\leqslant}
\renewcommand\geq{\geqslant}

\renewcommand\hat{\widehat}

\newcommand{\ft}[1]{\widehat #1}
\newcommand{\dotprod}[2]{\langle #1 , #2 \rangle}

\newcommand{\supp}{\operatorname{supp}}

\newcommand{\zft}[1]{\mathcal{Z}(\ft{\1}_{#1})}
\newcommand{\cm}{\complement}
\newcommand{\bd}[1]{\operatorname{bd}(#1)}
\newcommand{\half}{\tfrac{1}{2}}

\newcommand{\interior}{\operatorname{int}}

\newcommand{\convex}{\operatorname{conv}}
\newcommand{\relint}{\operatorname{relint}}
\newcommand{\aff}{\operatorname{aff}}

\theoremstyle{plain}
\newtheorem{thm}{Theorem}[section]
\newtheorem{lem}[thm]{Lemma}
\newtheorem{corollary}[thm]{Corollary}

\newtheorem*{claim*}{Claim}

\newcommand{\thmref}[1]{Theorem~\ref{#1}}
\newcommand{\secref}[1]{Section~\ref{#1}}

\newcommand{\lemref}[1]{Lemma~\ref{#1}}

\newcommand{\corref}[1]{Corollary~\ref{#1}}

\theoremstyle{definition}
\newtheorem{definition}[thm]{Definition}
\newtheorem*{definition*}{Definition}
\newtheorem*{remarks*}{Remarks}
\newtheorem*{remark*}{Remark}

\newtheorem{example}[thm]{Example}

\newenvironment{enumerate-roman}
{\begin{enumerate}
\addtolength{\itemsep}{5pt}
}
{\end{enumerate}}

\newenvironment{enumerate-alph}
{\begin{enumerate}
\addtolength{\itemsep}{5pt}
}
{\end{enumerate}}

\newenvironment{enumerate-num}
{\begin{enumerate}
\addtolength{\itemsep}{5pt}
}
{\end{enumerate}}

\newenvironment{enumerate-text}
{\begin{enumerate}
\addtolength{\itemsep}{5pt}
}
{\end{enumerate}}

\begin{document}

 \title{The Fuglede conjecture for convex domains is true in all dimensions}

\author{Nir Lev}
\address{Department of Mathematics, Bar-Ilan University, Ramat-Gan 5290002, Israel}
\email{levnir@math.biu.ac.il}

\author{M\'at\'e Matolcsi}
\address{Budapest University of Technology and Economics (BME),
H-1111, Egry J. u. 1, Budapest, Hungary (also at Alfr\'ed R\'enyi Institute of Mathematics,
Hungarian Academy of Sciences, H-1053, Realtanoda u 13-15, Budapest, Hungary)}
\email{matomate@renyi.hu}

\date{July 2, 2022}
\subjclass[2010]{42B10, 52B11, 52C07, 52C22}
\keywords{Fuglede's conjecture, spectral set, tiling, convex body, convex polytope}
\thanks{N.L.\ was supported by ISF Grants No.\ 227/17 and 1044/21 and ERC Starting Grant No.\ 713927.}
\thanks{M.M.\ was supported by NKFIH Grant No. K129335 and K132097.}

\begin{abstract}
A set $\Omega \subset \mathbb{R}^d$ is said to be spectral if the space $L^2(\Omega)$ has an orthogonal basis of exponential functions. A conjecture due to Fuglede (1974) stated that $\Omega$ is a spectral set if and only if it can tile the space by translations. While this conjecture was disproved for general sets, it has long been known that for a {\it convex body} $\Omega \subset \mathbb{R}^d$  the ``tiling implies spectral'' part of the conjecture is in fact true.

To the contrary, the ``spectral implies tiling'' direction of the conjecture for convex bodies was proved only in $\mathbb{R}^2$, and also in $\mathbb{R}^3$ under the a priori assumption that $\Omega$ is a convex polytope. In higher dimensions, this direction of the conjecture remained completely open (even in the case when $\Omega$ is a polytope) and could not be treated using the previously developed techniques.

In this paper we fully settle Fuglede's conjecture for convex bodies affirmatively in all dimensions, i.e.\ we prove that if a convex body $\Omega \subset \mathbb{R}^d$ is a spectral set then $\Omega$ is a convex polytope which can tile the space by translations. To prove this we introduce a new technique, involving a construction from crystallographic diffraction theory, which allows us to establish a geometric ``weak tiling'' condition necessary for a set $\Omega \subset \mathbb{R}^d$ to be spectral.
\end{abstract}

\maketitle


\section{Introduction} \label{secI1}

\subsection{}
Let $\Om \subset \R^d$ be a bounded, measurable set of positive
measure. We say that $\Om$ is \emph{spectral} if there exists a countable set 
$\Lambda\subset \R^d$  such that the system of exponential functions
\begin{equation}
	\label{eqI1.1}
	E(\Lambda)=\{e_\lambda\}_{\lambda\in \Lambda}, \quad e_\lambda(x)=e^{2\pi
	i\dotprod{\lambda}{x}},
\end{equation}
is orthogonal and complete in  $L^2(\Om)$,
  that is, the system constitutes an orthogonal basis in the space.
Such a set  $\Lambda$ is called a \emph{spectrum} for $\Om$.
The classical example of a spectral set is the unit cube $\Om = \left[-\frac1{2}, \frac1{2}\right]^d$, for which  the set $\Lam = \Z^d$ serves as a spectrum.

Which other sets $\Om$ can be spectral? The research on this problem
has been influenced for many years by a famous paper \cite{Fug74} due to Fuglede (1974),
who suggested that there should be a concrete, geometric way to characterize the
spectral sets. We say that $\Om$ \emph{tiles the space by translations} 
if there exists a countable set $\Lambda\subset \R^d$ such that the collection 
of sets $\{\Om + \lam\}$, $\lam \in \Lam$, consisting of translated copies of $\Om$,
constitutes a partition of $\R^d$ up to measure zero. In his paper,
Fuglede stated the following conjecture: 
``\emph{A set $\Om \subset \R^d$ is spectral if and only if it can tile the space by
translations}''.

For example, Fuglede proved that a triangle and a disk in
the plane are not spectral sets. He also proved that if 
$\Om$ can tile  with respect to a \emph{lattice} translation set $\Lambda$
then the dual lattice $\Lambda^*$ is a spectrum for $\Om$,
 and conversely. Fuglede's conjecture inspired extensive 
research over the years, and a number of interesting results establishing
connections between spectrality and tiling
had since been obtained.

The conjecture remained open for 30 years until a counterexample
was discovered by Tao \cite{Tao04}, who constructed in dimensions
5 and higher an example of a spectral set 
 which cannot tile by translations. 
Since then, counterexamples to both directions of the conjecture were
found in dimensions $d \geq 3$ (see \cite[Section 4]{KM10}).
 These examples are composed of finitely many
unit cubes in special arithmetic arrangements. 
The conjecture is still open in dimensions $d=1$ and $2$ 
in both directions.

On the other hand, it was believed that Fuglede's conjecture should be
true in all dimensions $d$ if the set $\Om \subset \R^d$
 is assumed to be a \emph{convex body}
 (that is, a compact convex set with nonempty interior\footnote{The
compactness and nonempty interior assumptions can be made with
no loss of generality, as any convex set of positive and finite
measure coincides with a convex body up to a set of measure zero.}).
Indeed, all the known counter\-examples to the conjecture are highly 
non-convex sets, being the union of a finite number of unit cubes centered at 
points of the integer lattice $\Z^d$. Moreover, it has long been known 
\cite{Ven54, McM80} that a convex body which tiles by
translations  must be a polytope, and that it 
admits a face-to-face tiling by a lattice translation set $\Lam$
and therefore  has a spectrum given by the dual lattice $\Lam^*$.
 So this implies that for a convex body $\Omega \subset \mathbb{R}^d$  
the ``tiling implies spectral'' part of the conjecture is in fact true in any dimension $d$.

To the contrary, the ``spectral implies tiling'' direction of the conjecture for convex bodies was proved only in $\mathbb{R}^2$ \cite{IKT03},  and also in $\mathbb{R}^3$ under the a priori assumption that $\Omega$ is a convex polytope \cite{GL17}. In higher dimensions, this direction of the conjecture remained completely open (even in the case when $\Omega$ is a polytope) and could not be treated using the previously developed techniques.

 It is our goal in the present paper to establish that the result in fact holds
in all dimensions and for general convex bodies.
We will prove the following theorem:

\begin{thm}
\label{thmA15}
Let $\Om$ be a convex body in $\R^d$. If $\Om$ is a spectral set,
then $\Om$ must be  a convex polytope, and it  tiles the
space face-to-face by translations along a lattice.
\end{thm}

This fully settles the Fuglede conjecture for  convex bodies 
affirmatively: we obtain that a convex body in $\R^d$  is a spectral set
if and only if  it can tile  by translations.

\subsection{}\label{secI1.2}
There is a complete characterization due to Venkov \cite{Ven54},
that was rediscovered by McMullen \cite{McM80, McM81}, of the convex
 bodies which tile by translations: 

\emph{A convex body $\Om \subset \R^d$ can tile the space
by translations if and only if it satisfies the following four conditions:
\begin{enumerate-num}
\item \label{vm:i} $\Om$ is a convex polytope; 
\item \label{vm:ii} $\Om$ is centrally symmetric; 
\item \label{vm:iii} all the facets of $\Om$ are centrally symmetric; 
\item \label{vm:iv} each belt of $\Om$ consists of either $4$ or $6$ facets. 
\end{enumerate-num}
Moreover, a convex body $\Om$ satisfying the four conditions
\ref{vm:i}--\ref{vm:iv} admits a face-to-face
tiling by translations along a certain lattice.}

We recall that 
 a \emph{facet} of a convex polytope $\Om \subset \R^d$ 
is a $(d-1)$-dimensional face of $\Om$. 
If $\Om$ has centrally  symmetric facets, then a \emph{belt} of 
$\Om$ is a system of facets obtained in the following
way: Let $G$ be a \emph{subfacet} 
of $\Om$, that is, a $(d-2)$-dimensional face.
Then $G$ lies in exactly two adjacent facets of $\Om$, say $F$ and $F'$.
Since $F'$ is centrally symmetric, there is another subfacet
$G'$ obtained by reflecting $G$ through the center of $F'$
(so in particular,  $G'$ is a translate of $-G$).
In turn, $G'$ is the intersection of $F'$ with another facet $F''$.
Continuing in this way, we obtain a system of facets $F, F', F'', \dots,
F^{(m)} = F$, called \emph{the belt of $\Om$ generated by the subfacet $G$},  
such that the intersection $F^{(i-1)} \cap F^{(i)}$ 
of any pair of consecutive facets in the system is a translate
 of either $G$ or $-G$. 

Fuglede's conjecture for convex bodies can thus be equivalently
stated by saying that for a  convex body $\Om \subset \R^d$ to
be spectral, it is necessary and sufficient that the four
conditions \ref{vm:i}--\ref{vm:iv} above hold.

In relation with the first condition \ref{vm:i}, a result proved in
 \cite{IKP99} states that if $\Om$ is a ball in $\R^d$ $(d \geq 2)$ 
then $\Om$ is not a  spectral set. In \cite{IKT01} this result was extended 
to the class of convex bodies $\Om \subset \R^d$ that have a smooth boundary.

As for \ref{vm:ii}, Kolountzakis \cite{Kol00} proved that 
if a convex body $\Om \subset \R^d$ is spectral, then it must 
be centrally symmetric. If $\Om$ is assumed a priori to be
a polytope, then another approach to this result
was given in \cite{KP02} (see also \cite[Section 3]{GL17}).

Recently, also the necessity of condition \ref{vm:iii} for spectrality
was established. It was proved in \cite[Section 4]{GL17} that
if a convex, centrally symmetric polytope $\Om \subset \R^d$
 is a spectral set, then all the facets of $\Om$
must also be centrally symmetric. The proof is based on a development 
of the argument in \cite{KP02}.

The last condition \ref{vm:iv}  was addressed so far only in dimensions
$d = 2$ and $3$. Iosevich, Katz and Tao proved in \cite{IKT03} 
that if a  convex polygon   $\Om \subset \R^2$ is a spectral set,  
then it must be either a parallelogram or a (centrally symmetric) 
hexagon. In three dimensions, it was recently proved \cite{GL16, GL17} 
that if a convex polytope $\Om \subset \R^3$ is  spectral,
then it can tile the space by translations (as a consequence, 
the condition  \ref{vm:iv} must hold, although the proof
does not establish  it directly).

In this paper, we will show that the conditions 
\ref{vm:i} and \ref{vm:iv} are in fact necessary  for the spectrality
of a general convex body $\Om$ 
in every dimension, thus obtaining a proof of the full Fuglede
 conjecture for convex bodies. Our main results are as follows:

\begin{thm}
\label{thmA20}
Let $\Om$ be a convex body in $\R^d$. If $\Om$ is a spectral set,
then $\Om$ must be  a convex polytope.
\end{thm}

\begin{thm}
\label{thmA21}
Let $\Om \subset \R^d$ be a convex polytope, which is centrally symmetric 
and has centrally symmetric facets. If $\Om$ is a spectral set,
then each belt of $\Om$ must consist of either $4$ or $6$ facets. 
\end{thm}

\thmref{thmA15} above follows as a consequence of 
these two theorems, combined with the results in \cite{Kol00} (or \cite{KP02}), 
\cite[Section 4]{GL17}, and the Venkov-McMullen theorem.

\subsection{}
In the above mentioned papers
 \cite{IKP99}, \cite{IKT01}, \cite{KP02},
\cite{IKT03},  \cite{GL17} the approach
relies on the asymptotic behavior of 
 the Fourier transform of the indicator function of  $\Om$,
and involves an analysis of its set of zeros.
In the present paper we introduce a new approach
to the problem, 
based on establishing a link between the notion of 
spectrality and a geometric notion which we refer to as ``weak tiling''.

\begin{definition}
Let $\Om \subset \R^d$ be a bounded, measurable set.
We say that another measurable, possibly unbounded, 
set $\Sig \subset \R^d$
admits a \emph{weak tiling} by translates of $\Om$,
if there exists a positive, locally finite  (Borel) measure $\mu$
on $\R^d$ such that $\1_{\Om} \ast \mu = \1_{\Sig}$ a.e.
\end{definition}

If the measure $\mu$ is the sum of unit masses at the points of a locally
finite set $\Lambda$, that is,  $\mu = \sum_{\lam \in \Lam} \delta_\lam$, then
the condition $\1_{\Om} \ast \mu = \1_{\Sig}$ a.e.\ means
that the collection 
$\{\Om + \lam\}$, $\lam \in \Lam$, of translated copies of $\Om$,
constitutes a partition of $\Sig$ up to measure zero. 
In this case, we say that the weak tiling is a \emph{proper tiling}.

For example, one can check that any bounded set $\Om$ of positive 
Lebesgue measure  tiles the  whole space  $\R^d$ weakly by translates 
with respect to the measure $d\mu  = m(A)^{-1} \, dx$ (where $dx$ denotes
 the Lebesgue measure on $\R^d$). This is in sharp contrast to the obvious fact
that not every set $\Om$ can tile the space 
properly by translations.

We will prove the following theorem, which gives a necessary
 condition for spectrality in terms of weak tiling:

\begin{thm}
\label{thmA11}
Let $\Om$ be a bounded, measurable set  in $\R^d$. If $\Om$ is spectral,
then its complement $\Om^\cm = \R^d \setminus \Om$
 admits a weak tiling by translates of $\Om$. That is,
 there exists a positive, locally finite measure $\mu$ such that
$\1_{\Om} \ast \mu = \1_{\Om^\cm}$ a.e.
\end{thm}

Our proof of this result involves a construction due to Hof \cite{Hof95},
that is often used in mathematical crystallography in order 
to describe the diffraction pattern of an atomic structure
(see also \cite[Chapter 9]{BG13}).

Notice that if
 the complement $\Om^\cm$ has a \emph{proper} tiling by translates
 of $\Om$, then it just means that $\Om$ can tile the space by translations. 
\thmref{thmA11} thus establishes a weak form of the
``spectral implies tiling'' part of  Fuglede's conjecture, which is valid 
for all bounded, measurable sets $\Om \subset \R^d$. 
We observe 
that the weak tiling conclusion cannot be strengthened to
proper tiling without imposing extra assumptions on the set $\Om$,
since there exist examples of spectral sets which cannot tile
by translations.

We will prove  that if $\Om$ is a convex body in $\R^d$, and if
it can tile its complement $\Om^\cm$ weakly by 
translations, then $\Om$ must in fact be a convex polytope (\thmref{thmE5}).
We will also prove that  if, in addition,
 $\Om$ is centrally symmetric and has centrally symmetric facets, 
then each belt of $\Om$  must have either $4$ or $6$ facets 
(\thmref{thmJ1.0}). So in the latter case, it follows that $\Om$
can in fact tile its complement $\Om^\cm$ not only weakly, but even
 properly, by translations.

The potential applications of \thmref{thmA11} are not limited 
to the class of convex bodies in $\R^d$.
As an  example,  we will use this theorem
to give a simple geometric condition which is
necessary for the spectrality of a bounded,
measurable set $\Om \subset \R^d$
(\thmref{thmA2}). Based on this
condition we will  prove that the boundary of a bounded,
open spectral set 
must have Lebesgue measure zero (\thmref{thmB1}).

\subsection{}
The rest of the paper is organized as follows.

In \secref{sect:prelim} we present some preliminary background.
We fix notation that will be used in the paper and discuss
basic results about measures and tempered distributions,
spectral sets and weak tiling.

In \secref{sect:diffract} we prove
 that if  a bounded, measurable set  
$\Om \subset \R^d$ is spectral, then its complement $\Om^\cm$
 admits a weak tiling by translates of $\Om$
(\thmref{thmA11}).  As an application we
show that a connected spectral domain cannot
have any ``holes'',
and that the boundary of an 
open spectral domain must have Lebesgue
 measure zero.

 In \secref{sect:polytope} we prove that if a convex body 
$K \subset \R^d$ is a spectral set,
then $K$ must be  a convex polytope (\thmref{thmA20}).

In the last two Sections \ref{sect:belti} and \ref{sect:beltii} we establish that
each belt of a spectral convex polytope 
$K \subset \R^d$ 
 must consist of either $4$ or $6$ facets (\thmref{thmA21}).
The proof is based on an analysis 
 of the measure $\mu$ 
that provides a weak tiling of $K^\cm$
by translates of $K$.


\section{Preliminaries}
\label{sect:prelim}

\subsection{Notation}
If $A \subset \R^d$ then $\interior(A)$ will denote the 
interior of $A$, and $\bd{A}$ the boundary of $A$.
We use  $A^\cm$ to denote the 
complement   $\R^d \setminus A$ of the set $A$.
If $A,B \subset \R^d$ then
$A \triangle B$ is the symmetric difference
of $A$ and $B$.
We denote by $|A|$ the number of elements in $A$.

If $A \subset \R^d$, then for each $\tau \in \R^d$
we let $A + \tau = \{a+ \tau: a \in A\}$ denote
 the image of $A$ under translation by the vector $\tau$. If $s \in \R$,
then $sA = \{s a : a \in A\}$ will denote the image of $A$ 
under dilation with ratio $s$. If $A, B $ are two subsets of $\R^d$, then
$A+B$ and $A-B$ denote respectively their set of sums and set of 
differences.

We use $\dotprod{\cdot}{\cdot}$ and $|\cdot|$ to denote
the Euclidean  scalar product and norm in $\R^d$.

By a \emph{lattice} in $\R^d$ we mean a set $L$ which can
be obtained as  the image of $\Z^d$ under an
invertible linear transformation. The \emph{dual lattice} $L^*$ is the set of all 
vectors $\lambda^* \in \R^d$ such that $\dotprod{\lambda}{\lambda^*}
 \in \Z$ for every $\lambda \in L$.

We denote by $m(A)$ the Lebesgue measure of a set $A \subset \R^d$. 
We also use  $m_k(A)$ to denote the $k$-dimensional volume measure
of $A$ (so, in particular, $m_d(A) = m(A)$).

If $A \subset \R^d$ is a bounded, measurable set, then we define
\[
	\Delta(A) := \{x \in \R^d: m(A \cap (A + x)) > 0\}.
\]
Then $\Delta(A)$ is a bounded open set, and 
 we have $\Delta(A) = - \Delta(A)$
(which means that $\Delta(A)$ is 
 symmetric with respect to the origin).
One can think of the set $\Delta(A)$ as the measure-theoretic analog of the
set of differences $A-A$. In particular, one can check that if $A$ is an open 
set then  $\Delta(A) = A-A$.
In general we  have $\Delta(A) \subset A-A$, but this inclusion can be strict.

The Fourier transform of a function $f \in L^1(\R^d)$ is defined by
\[
\ft f (t)=\int_{\R^d} f (x) \, e^{-2\pi i\langle t,x\rangle} dx.
\]

\subsection{Measures and distributions}
By a ``measure'' we will refer to a Borel (either positive, or complex)
measure on $\R^d$.
We use $\supp(\mu)$ to denote the closed support of a measure $\mu$.
We denote by $\delta_\lam$ the Dirac measure consisting of a unit
mass at the point $\lam$. If $\Lam \subset \R^d$  is   a countable set,
then we define $\delta_\Lam := \sum_{\lam \in \Lam} \delta_\lam$.

If $\alpha$ is a tempered distribution on $\R^d$,
 and if $\varphi$ is a Schwartz function on  $\R^d$, then we use
$\dotprod{\alpha}{\varphi}$ to denote the action of $\alpha$ on
 ${\varphi}$. 
A tempered distribution $\alpha$ is  \emph{positive} if  we have
$\dotprod{\alpha}{\varphi} \geq 0$
for any Schwartz function $\varphi \geq 0$.
 If a tempered distribution  $\alpha$
 is  positive, then $\alpha$ is a positive measure. 
The Fourier transform $\ft\alpha$ of a tempered distribution $\alpha$ is defined by 
$\dotprod{\ft\alpha}{\varphi} = \dotprod{\alpha}{\ft\varphi}$.
A tempered distribution 
$\alpha$ is said to be \emph{positive-definite} if $\ft\alpha$ is a positive 
distribution. See \cite[Section 8.4]{BG13}, \cite{Rud91}.

If $\mu$ is a  measure on $\R^d$, then $\mu$ is said to be 
\emph{locally finite} if we have $|\mu|(B) < \infty$ for every open ball $B$.

A measure $\mu$ on $\R^d$  is said to be \emph{translation-bounded}
 if for every (or equivalently, for some) open ball $B$ we have 
\[
\sup_{x \in \R^d} |\mu|(B+x) < \infty.
\]
If a measure $\mu$ on $\R^d$ is  translation-bounded,
then it is a tempered distribution. 
If $\mu$  is a translation-bounded measure  on $\R^d$, and if $\nu$
is a finite measure  on $\R^d$, then the convolution $\mu \ast \nu$
is a translation-bounded measure.

\begin{lem}
\label{lemB1}
Let $\nu$ be a finite measure on $\R^d$, and suppose that
$\mu$ is a translation-bounded measure on $\R^d$
whose Fourier transform $\ft{\mu}$ is a locally finite measure.
Then the Fourier transform of the convolution $\mu \ast \nu$
 is the measure $\ft{\mu} \cdot \ft{\nu}$.
\end{lem}

See e.g.\ \cite[Section 8.6]{BG13}, \cite[Section 2.5]{KL21}.

A sequence of measures $\{\mu_n\}$ is said to be
\emph{uniformly translation-bounded} if 
for every (or equivalently, for some) open ball $B$ one can find a constant $C$ 
not depending on $n$, such that
$\sup_x |\mu_n|(B+x) \leq C$ for every $n$.

If $\{\mu_n\}$ is a uniformly translation-bounded sequence of measures,
then we say that $\mu_n$ \emph{converges vaguely} to a measure $\mu$ if for
every continuous, compactly supported function $\varphi$ we have
$\int \varphi \, d\mu_n \to \int \varphi \, d\mu$. In this case, the 
vague limit  $\mu$
must also be a translation-bounded measure. For a uniformly translation-bounded 
sequence of measures  $\{\mu_n\}$  to converge vaguely,
it is necessary and sufficient that $\{\mu_n\}$  converge
 in the space of tempered distributions.
From any uniformly translation-bounded sequence of 
measures $\{\mu_n\}$  one can extract
a vaguely convergent subsequence $\{\mu_{n_j}\}$.

\begin{lem}
\label{lemC2}
Let $f \in L^1(\R^d)$,  and let $\{\mu_n\}$ be a uniformly
 translation-bounded sequence of measures
on $\R^d$, such that $f \ast \mu_n = 1$ a.e.\ for every $n$. 
If $\mu_n$ converges vaguely to a measure $\mu$ then also
$f \ast \mu = 1$ a.e.
\end{lem}

\begin{proof}
Let $\varphi$ be a continuous, compactly supported function on $\R^d$.
Then the sequence of functions $ \mu_n \ast \varphi $ is uniformly bounded and
converges pointwise to $\mu \ast \varphi$. Hence by the 
dominated convergence theorem we have $f \ast 
(\mu_n \ast \varphi) \to f \ast (\mu \ast \varphi)$ pointwise.
In turn, this implies that $(f \ast \mu_n) \ast \varphi \to (f \ast \mu) \ast \varphi$ 
pointwise, since the convolution is associative  (by Fubini's theorem).
But $f \ast \mu_n = 1$ a.e.\ for every $n$, so  we conclude that
$(f \ast \mu) \ast \varphi = \int \varphi$. Since this is true
for an arbitrary $\varphi$, the assertion follows.
\end{proof}

\subsection{Spectra}
If $\Om$ is a bounded, measurable set in $\R^d$ of positive measure, then
by a \emph{spectrum} for $\Om$ we mean a countable set $\Lambda\subset\R^d$
such that the system of exponential functions $E(\Lam)$ defined by \eqref{eqI1.1} is 
orthogonal and complete in the space $L^2(\Om)$.

For any two points $\lam,\lam'$ in $\R^d$ we have
$\dotprod{e_\lambda}{e_{\lambda'}}_{L^2(\Om)} = \hat{\1}_\Om(\lambda'-\lambda)$, where
$\ft{\1}_\Om$ is the Fourier transform of the indicator function $\1_\Om$ of the set $\Om$.
The orthogonality of the system $E(\Lambda)$ in $L^2(\Om)$ is therefore equivalent to the condition
\begin{equation}
	\label{eqP1.2}
	\Lambda-\Lambda \subset  \zft{\Om} \cup \{0\}, 
\end{equation}
where $\zft{\Om} := \{ t \in \R^d : \hat{\1}_\Om(t)=0\}$ is the set of 
zeros of the function $\hat{\1}_\Om$.

A set $\Lambda\subset \R^d$ is said to be \emph{uniformly discrete} if there is
$\delta>0$ such that $|\lambda'-\lambda|\ge \delta$ for any two distinct points
$\lambda,\lambda'$ in $\Lambda$. 
The condition \eqref{eqP1.2} implies that every spectrum $\Lambda$ of $\Om$
is a uniformly discrete set.

The set $\Lambda$ is  \emph{relatively dense} if
there is $R >0$ such that every ball of radius $R$ contains at least one point from $\Lam$.
 It is well-known that if $\Lam$ is a spectrum for $\Om$,
then $\Lam$ must be a relatively dense 
set (see e.g.\ \cite[Section 2C]{GL17}).

The following lemma gives a frequently used characterization of the
spectra of $\Om$.

\begin{lem}[{see \cite[Section 3.1]{Kol04}}]
	\label{lemC1.7}
	Let $\Om$  be a bounded, measurable set in $\R^d$, and define
	the function $f := m(\Om)^{-2} \, |\ft{\1}_\Om|^2$. 
	Then a set $\Lam\subset\R^d$ is a spectrum for $\Om$ 
	if and only if $f \ast \delta_\Lam = 1$ a.e.
\end{lem}

\subsection{Weak tiling}
If  $\Om$ is a bounded, measurable set in $\R^d$, and if
$\Sig$ is another, possibly unbounded, measurable  set in $\R^d$,
then we say that $\Sig$
admits a \emph{weak tiling} by translates of $\Om$,
if there exists a positive, locally finite measure $\mu$
on $\R^d$ such that $\1_{\Om} \ast \mu = \1_{\Sig}$ a.e.

The following lemma shows that the measure $\mu$
in any weak tiling is not only locally finite, but in fact 
 must be  translation-bounded.

\begin{lem}
\label{lemA8.7}
Let $\Om \subset \R^d$ be a bounded,
 measurable set of positive Lebesgue measure,
and suppose that $\mu$ is a positive, locally finite measure such that
$\1_\Om \ast \mu \leq 1$ a.e. Then $\mu$ is a
 translation-bounded measure.
\end{lem}

\begin{proof}
Fix $r>0$, and let  $B_r$ be the open ball of radius
$r$ centered at the origin.
It will be enough to show that there is a constant
$M > 0$ such that $(\1_{B_r} \ast \mu)(x) \leq M$ for every 
$x \in \R^d$.  Since $\Om$ is a bounded set,  we can choose
a sufficiently  large number $s>0$ such that we have 
$B_s + y \supset B_r$ for every $y \in \Om$.  It follows that
$\1_{B_s} \ast \1_\Om \geq m(\Om) \, \1_{B_r}$. In turn,
this implies that
\[
m(\Om) \, \1_{B_r} \ast \mu \leq ( \1_{B_s} \ast \1_\Om ) \ast \mu
=  \1_{B_s} \ast (\1_\Om  \ast \mu)
\leq \1_{B_s} \ast 1 = m(B_s).
\]
We thus see that the constant $M := m(B_s) / m(\Om)$
satisfies $\1_{B_r} \ast \mu \leq M$  as needed.
\end{proof}

The next lemma implies that if $\Om$ tiles $\Sig$ weakly
by translations with respect to a measure $\mu$, then
we must have $m((\Om + t) \cap \Sig^\cm) = 0$ for
every $t \in \supp(\mu)$.

\begin{lem}
\label{lemA8.2}
Let $\Om$  be a bounded,  measurable set in $\R^d$, and let 
$\mu$ be a positive, locally finite measure on $\R^d$. Suppose that we have 
$\1_\Om \ast \mu = 0$ a.e.\ on the complement $\Sig^\cm$ of
another measurable set $\Sig \subset \R^d$. 
Then $m((\Om+t) \cap \Sig^\cm) = 0$
 for every $t \in \supp(\mu)$.
\end{lem}

\begin{proof}
Let $\varphi(t) := m((\Om+t) \cap \Sig^\cm)$, $t \in \R^d$.
Suppose to the contrary that $\varphi(t) > 0$ for some
$t \in \supp(\mu)$.  Since $\varphi$  is a continuous function, 
there is an open neighborhood $U$ of $t$ such that
$\varphi(t') > 0$ for all $t' \in U$. Since $t \in \supp(\mu)$ we must have $\mu(U)> 0$,
and it follows that $\int \varphi \, d\mu > 0$. But on the other hand,
using Fubini's theorem we have
\[
\int \varphi \, d\mu 
= \int_{\Sig^\cm}  ( \1_\Om \ast \mu ) \, dm = 0,
\]
since we have assumed that
$\1_\Om \ast \mu = 0$ a.e.\ on $\Sig^\cm$.
We thus  obtain a contradiction.
\end{proof}

\begin{corollary}
\label{corA8.1}
Let $\Om$ be a bounded, measurable set  in $\R^d$. Assume that the
complement $\Om^\cm$ of $\Om$ admits a weak tiling by translates of $\Om$,
that is, there is a positive, locally finite measure $\mu$ such that
$\1_\Om \ast \mu = \1_{\Om^\cm}$ a.e. 
Then $\supp(\mu) \subset \Delta(\Om)^\cm$.
\end{corollary}

Indeed, this follows from \lemref{lemA8.2} in the special case when $\Sig = \Om^\cm$.


\section{Spectrality and weak tiling}
\label{sect:diffract}

In this section, our main goal is to prove 
 \thmref{thmA11}, which states that if  a bounded, measurable set  
$\Om \subset \R^d$ is spectral, then its complement $\Om^\cm$
 admits a weak tiling by translates of $\Om$. 
The proof of this result involves a construction due to Hof \cite{Hof95},
that is often used in mathematical crystallography in order 
to describe the diffraction pattern of an atomic structure.

We also present some simple applications of 
 \thmref{thmA11}. As an example, we will use it
to prove that if a bounded, open set  $\Om \subset \R^d$
  is spectral, then its boundary must have Lebesgue measure zero.
Some other  examples  will also be discussed.

\subsection{}
In  \cite{Hof95}, Hof suggested a mathematical model for 
the  diffraction experiment used  in crystallography  to analyze the
atomic structure of a solid. In this model, the configuration of the
atoms is represented by a uniformly discrete  and relatively
dense set $\Lam \subset \R^d$.
Hof used volume averaged convolution of the two measures 
$\delta_\Lam$ and $\delta_{-\Lam}$  in order
to define the \emph{autocorrelation measure} $\gamma$ of the set $\Lam$.
He then argued that the diffraction by the atomic
structure is described by the Fourier
transform $\ft \gam$ of the autocorrelation, which is a positive measure
called the \emph{diffraction measure} of the set $\Lam$
(for more details, see \cite[Chapter 9]{BG13}).

Here, we will apply this technique to a set $\Lam$ that
constitutes a spectrum for a bounded, measurable set
$\Om \subset \R^d$. We will obtain the following result:

\begin{thm}
\label{thmA8}
Let $\Om$ be a bounded, measurable set  in $\R^d$.
If $\Om$ is a spectral set, then there exists a measure
$\gamma$ on $\R^d$ with the following properties:
\begin{enumerate-alph}
	\item \label{thmA8.1}
	$\gamma$ is a positive, translation-bounded measure;
	\item \label{thmA8.2}
	the support of $\gamma$ is contained in the closed set $\zft{\Om} \cup \{0\}$;
	\item \label{thmA8.3}
	$\gamma = \delta_0$ in some open neighborhood of the origin;
	\item \label{thmA8.4}
	$\ft{\gamma}$ is also a positive, translation-bounded measure;
	\item \label{thmA8.5}
	$\ft{\gamma} = m(\Om) \, \delta_0$ in the open set $\Delta(\Om)$.
\end{enumerate-alph}
\end{thm}

A similar result can be proved in the more general context of locally compact 
abelian groups, but in this paper we work in the euclidean setting only.

We remark that a
 minor difference in our proof compared to the construction
in  \cite{Hof95}, is that we  define  the autocorrelation measure $\gam$ 
not as the volume averaged convolution of the two measures 
$\delta_\Lam$ and $\delta_{-\Lam}$, but instead as
a convolution averaged with respect to the 
number of points in $\Lam$. It turns out that using
this normalization allows in our context to establish more directly the desired properties
of the measure $\gam$.

\begin{proof}[Proof of \thmref{thmA8}]
Suppose that $\Om$ is a spectral set, and let
 $\Lam$ be a spectrum for $\Om$. Then $\Lam$ is a uniformly discrete 
and relatively dense set in $\R^d$. For  each $r>0$ we denote
$\Lam_r := \Lam \cap B_r$, where $B_r$ is the open ball of radius
$r$ centered at the origin. Then there is $r_0$ such that $\Lam_r$
is nonempty for every $r \geq r_0$. Consider a family of measures
$\{\gamma_r\}$, $r \geq r_0$, defined by
$\gamma_r := |\Lam_r|^{-1} \, \delta_{\Lam_r} \ast \delta_{-\Lam_r}$.
Each $\gamma_r$ is a positive, finite
 measure on $\R^d$, whose support is the finite set
$\Lam_r - \Lam_r$. We have
$\ft{\gamma}_r(t) = |\Lam_r|^{-1} \, |\ft{\delta}_{\Lam_r}(t)|^2 \geq 0$,
hence $\gamma_r$ is a positive-definite measure. Since
$\Lam$ is  uniformly discrete, the measures
$\gamma_r$ are uniformly translation-bounded, namely,
for any open ball $B$ we have
$\sup_x \gamma_r(B+x) \leq C(\Lam,B)$,
where $C(\Lam,B)$ is a constant which does not depend on $r$.
It follows that we may choose a sequence $r_n \to \infty$ such that
if we define $\mu_n := \gamma_{r_n}$, then the sequence
$\{\mu_n\}$ converges vaguely  to some measure $\gamma$ on $\R^d$.

The  measure $\gamma$  is 
 translation-bounded, positive and positive-definite. It follows from
the positive-definiteness of $\gamma$ that
 its Fourier transform $\hat{\gamma}$ is a positive measure.
Given any open ball $B$, let $\varphi_B$ be a Schwartz function  
such that $\varphi_B \geq \1_B$. Then we have
\[
\ft{\gamma}(B+x)  \leq \int \varphi_B(y-x) \, d \, \ft{\gamma}(y) =
  \int e^{-2 \pi i \dotprod{x}{t}} \, \ft{\varphi}_B(t) \, d\gamma(t),
\]
and hence
\[
\sup_{x \in \R^d} 
\ft{\gamma}(B+x)   \leq 
\int |\ft{\varphi}_B(t)| \, d\gamma(t) < \infty
\]
(the last integral is finite since $\ft{\varphi}_B$ has fast
decay, while $\gamma$ is  translation-bounded).
We conclude that $\ft{\gamma}$ is a translation-bounded measure. 

Define now another sequence of measures $\{\nu_n\}$ by
$\nu_n := |\Lam_{r_n}|^{-1} \, \delta_{\Lam} \ast \delta_{-\Lam_{r_n}}$.
The measures $\nu_n$ are uniformly translation-bounded, again due to the
uniform discreteness of $\Lam$. We claim that the sequence  
$\{\nu_n\}$ converges vaguely to the same limit as
the sequence  $\{\mu_n\}$, namely, to the measure $\gamma$.
To see this, it will be enough to check that the sequence of differences
$\{\nu_n - \mu_n\}$ converges vaguely to zero. Indeed, we have
\[
\nu_n - \mu_n=
|\Lam_{r_n}|^{-1} \, \delta_{\Lam \setminus \Lam_{r_n}} \ast \delta_{-\Lam_{r_n}}.
\]
For any fixed $s >0$, the total mass of the measure
$\nu_n - \mu_n$ in the open ball $B_s$ is equal to 
$|\Lam_{r_n}|^{-1}$ times the number of
pairs $(\lam,\lam') \in \Lam\times \Lam$ such that
$|\lam|<r_n$, $|\lam'|\geq r_n$ and $|\lam'-\lam|<s$.
As $n \to \infty$, the number of such pairs is not greater than a constant multiple
of $r_n^{d-1}$ since $\Lam$ is a uniformly discrete set,
while the number of elements in the set
$\Lam_{r_n}$ is not less than a constant multiple of
$r_n^d$ since $\Lam$ is also relatively dense. Hence $|\nu_n - \mu_n|(B_s)$ 
tends to zero as $n \to \infty$. Since this is true for any $s>0$, it follows 
that the sequence $\{\nu_n - \mu_n\}$ converges to zero, 
and so the sequence $\nu_n$ converges to $\gamma$.

Let $f := m(\Om)^{-2} \, |\ft{\1}_\Om|^2$.
Since $\Lam$ is a spectrum for $\Om$,  by \lemref{lemC1.7}
we have $f \ast \delta_\Lam = 1$ a.e.
As $\nu_n$ is an average of translates of the measure $\delta_\Lam$,
we also have  $f \ast \nu_n = 1$ a.e.\ for every $n$. Since 
the measures $\nu_n$ are uniformly translation-bounded
and converge vaguely to $\gamma$, it follows 
from \lemref{lemC2} that  $f \ast \gamma = 1$ a.e.\ as well.
In turn, using \lemref{lemB1} this implies that 
$\ft{f} \cdot \ft{\gamma} = \delta_0$.
Since $\ft{f}(0) = \int f = m(\Om)^{-1}$, we deduce that 
$\ft{\gamma} =  m(\Om) \, \delta_0$ in the 
open set $\{t : \ft{f}(t) \neq 0\}$. But the Fourier transform of $f$ is the
function  $\ft{f} = m(\Om)^{-2} \, \1_\Om \ast \1_{-\Om}$, that is, we have
$\ft{f}(t) = m(\Om)^{-2} \, m(\Om \cap (\Om+t))$ for every $t \in \R^d$.
 Hence $\ft{f}(t) \neq 0$ if and only if $t \in \Delta(\Om)$, and we conclude
that $\ft{\gamma} =  m(\Om) \, \delta_0$ in the open set $\Delta(\Om)$.

Since $\Lam$ is a uniformly discrete set, there exists $\eps > 0$ such that
$(\Lam - \lam) \cap B_\eps = \{0\}$
for every $\lam \in \Lam$.  In other words, for every $\lam \in \Lam$ we have
$\delta_\Lam \ast \delta_{-\lam} = \delta_0$ in the ball $B_\eps$.
 The measure $\nu_n$ is an average of measures of
the form $\delta_\Lam \ast \delta_{-\lam}$ $(\lam \in \Lam)$, so we also
have $\nu_n = \delta_0$ in the ball $B_\eps$. It follows that
the same is
true for the vague limit $\gamma$ of the sequence $\nu_n$, that is, we have
$\gamma = \delta_0$ in the open ball $B_\eps$.

The supports of all the measures $\mu_n$ are contained in
the set of differences
$\Lam-\Lam$. Since $\Lam$ is a spectrum for $\Om$, the set 
$\Lam-\Lam$ is contained in the closed  set $\zft{\Om} \cup \{0\}$.
The measure $\gamma$ is the vague limit of the sequence $\mu_n$, 
so it follows that the closed support of $\gamma$ is also contained
in the set $\zft{\Om} \cup \{0\}$. The theorem  is thus proved.
\end{proof}

\subsection{}
We can now use \thmref{thmA8} to deduce  \thmref{thmA11}.

\begin{proof}[Proof of \thmref{thmA11}]
Let $\Om$ be a bounded, measurable set  in $\R^d$, and assume
that $\Om$ is spectral. We must show that
$\Om^\cm$  admits a weak tiling by translates of $\Om$.

Indeed, let  $\gamma$ be the measure given by \thmref{thmA8}.
We have $\ft{\1}_{- \Om} (t) = \overline{\hat{\1}_{\Om}(t)}$,
and hence  $\supp(\gam)$ is contained in the set $\zft{-\Om} \cup \{0\}$. We also have
$\gamma = \delta_0$ in some open neighborhood of the origin.
 Since $\ft{\1}_{-\Om}(0) = m(\Om)$, it follows  that
$\gamma \cdot \hat{\1}_{-\Om} = m(\Om) \, \delta_0$.

The measure $\gam$ is translation-bounded, and its Fourier
transform $\ft{\gamma}$ is a translation-bounded measure as well.
Hence by \lemref{lemB1}, the Fourier
transform of the measure $\gamma \cdot \hat{\1}_{-\Om}$ is the function
$\ft{\gamma} \ast \1_{\Om}$. We thus conclude that
$\hat{\gamma} \ast \1_\Om = m(\Om)$ a.e.

The measure $\ft{\gamma}$ is positive, and satisfies 
$\ft{\gamma} = m(\Om) \, \delta_0$ in the open set $\Delta(\Om)$.
We can therefore write $ \ft{\gamma} =  m(\Om) \, ( \delta_0 + \mu )$,
where $\mu$ is a  positive, translation-bounded (and hence locally finite)
measure. The condition
$\hat{\gamma} \ast \1_\Om = m(\Om)$ a.e.\ is then equivalent to
$\1_\Om \ast \mu = \1_{\Om^\cm}$ a.e., and we obtain
that $\Om^\cm$  has a weak tiling by translates of $\Om$.
\end{proof}

\subsection{}
To illustrate the  construction of the autocorrelation measure
 $\gamma$ in the proof of \thmref{thmA8}, and  the corresponding weak tiling 
of $\Om^\cm$ obtained in \thmref{thmA11}, we now look at a few examples.

\begin{example} \label{expA8.1}
Assume that $\Om \subset \R^d$ tiles the space with respect
to a \emph{lattice}  translation set $L$. Let $\Lambda$ be a spectrum of $\Om$
 given by the dual lattice, that is, $\Lambda=L^*$. Then 
$\Lam - \lam = \Lam$ for every $\lam \in \Lam$, and hence the measures 
$\nu_n := |\Lam_{r_n}|^{-1} \, \delta_{\Lam} \ast \delta_{-\Lam_{r_n}}$,
whose vague limit is the 
autocorrelation measure $\gam$, satisfy  $\nu_n = \delta_\Lam$
for every $n$. We conclude that the autocorrelation measure $\gam$ is given by
$\gamma = \delta_\Lam = \delta_{L^*}$. In turn, by the Poisson
summation formula we  get 
$\ft{\gamma} = m(\Om) \, \delta_{L}$, and so the weak
tiling measure $\mu$ in \thmref{thmA11} is given by
$\mu = \delta_{L \setminus \{0\}}$. We thus see that
the weak tiling of $\Om^\cm$  is in this case a proper tiling.
\end{example}

\begin{example} \label{expA8.2}
Let $\Om = [0, \tfrac1{2}] \cup [1, \tfrac{3}{2}]$. Then
 $\Om$ tiles the real line $\R$ properly with respect to the 
translation set $\Lam = 2\Z \cup (2\Z + \tfrac1{2})$.
It is  well-known and not difficult to verify that the same
set $\Lam$ is also a spectrum for $\Om$. The set $\Lam$ is
periodic, but it is not a lattice. We calculate the autocorrelation 
measure $\gam$ as the vague limit of the measures
$\nu_n := |\Lam_{n}|^{-1} \, \delta_{\Lam} \ast \delta_{-\Lam_{n}}$,
where $\Lam_n := \Lam \cap [-2n, 2n)$. Then  all the measures
$\nu_n$, as well as the vague limit $\gam$, coincide with
the measure
$\sum_{k \in \Z} \cos^2( \tfrac{\pi k}{4} )\delta_{k/2}$
(so that, in fact, $\nu_n$ does not depend on $n$).
In turn, one can check using the Poisson summation formula that
$\ft \gam = \gam$. We thus obtain that the weak tiling
 measure $\mu$ in \thmref{thmA11} is a pure point measure 
(that is, $\mu$ has no continuous part), but  the
weak tiling of $\Om^\cm$ with respect to this measure $\mu$
  is not a proper tiling.
\end{example}

\begin{example} \label{expA8.3}
Let $\Om = [- \tfrac1{2}, \tfrac1{2}]^2$ be the unit cube in $\R^2$,
and let $\alpha$ be an irrational real number. Define $\Lambda$ to 
be the set of all points of the form 
$(n, n^2 \alpha + m)$ where $n$ and $m$ are integers.
It is easy to see that $\Om$ tiles the plane $\R^2$ with respect to the
translation set $\Lambda$. It is known, see \cite{JP99},
that the  translation sets for tilings by the unit cube
 $\Om$ coincide with the spectra of $\Om$.
Hence $\Lam$ is a spectrum for $\Om$.

In this example, we will calculate the autocorrelation measure 
$\gam$ as the vague limit of the measures
$\nu_N := |\Lam_{N}|^{-1} \, \delta_{\Lam} \ast \delta_{-\Lam_{N}}$,
where $\Lam_N := \Lam \cap Q_N$ is the intersection of $\Lam$ with
the cube $Q_N := [-N,N)^2$, that is, we will be averaging not with
respect to balls, but instead with respect to cubes
with sides parallel to the axes. Notice that the proof of \thmref{thmA8}
remains valid if the averages are taken with respect to cubes instead of
balls (in fact, Hof used cubes and not balls in \cite{Hof95}),
 although  the autocorrelation measure $\gam$ could, in principle,
depend on the shape over which the average is taken.

It follows from the fact that
 $\Lam - \Lam \subset \Z \times \R$ that all the measures
$\nu_N$, and hence also 
 the autocorrelation measure $\gam$, are  supported on
$\Z \times \R$. The restriction  of the
measure $\nu_N$ to a line of the form $\{h \} \times \R$,
where $h \in \Z$, is given by
\[
\frac{1}{2N} \sum_{-N \leq n < N} \sum_{k \in \Z}
\delta_{(h,  2  h n \alpha + h^2 \alpha + k).}
\]
Since $\alpha$ is irrational, it follows from Weyl's equidistribution theorem 
that $\nu_N$ converges vaguely to the one-dimensional Lebesgue measure
on each  line $\{h \} \times \R$ such that $h \neq 0$. On  the other hand, 
on the line $\{ 0 \} \times \R$ all the measures $\nu_N$ coincide
with $\delta_{\{ 0 \} \times \Z}$. We conclude that
$\gamma = \delta_0 \times \delta_{\Z} + \delta_{\Z \setminus \{0\}} \times m_1$,
where $m_1$ denotes the one-dimensional Lebesgue measure.
In turn, the diffraction measure $\ft \gam$ is given by
$\ft \gam =  \delta_{\Z}  \times \delta_0 + m_1 \times \delta_{\Z \setminus \{0\}}$.
So in this example, we obtain that
 the weak tiling of $\Om^\cm$ given  by \thmref{thmA11}
involves a measure $\mu$ that has both a pure point part
and a (singular) continuous part.
\end{example}

\subsection{}
Our main application of \thmref{thmA11} will be for the proof of 
Fuglede's conjecture for convex bodies in $\R^d$. This will be done in
the following sections. But before that, we mention some
simple applications of the theorem to  other classes of domains.

The following result, for instance, excludes many sets from being spectral:

\begin{thm}
\label{thmA2}
Let $\Om$ be a bounded, measurable set  in $\R^d$. Suppose that there exists a measurable set $S \subset \R^d$ with the following properties:
\begin{enumerate-roman}
	\item \label{thmA2.1}
	$m(S)>0$;
	\item \label{thmA2.2}
	$m(S \cap \Om)=0$;
	\item \label{thmA2.3}
	for every $x \in \R^d$, if $m((\Om+x) \cap S)>0$ then $m((\Om+x) \cap \Om)>0$.
\end{enumerate-roman}
Then $\Om$ cannot be a spectral set.
\end{thm}

Informally speaking, the assumption in this theorem means
 that there exists a portion $S$ of  the complement $\Om^\cm$ of $\Om$, 
such that no translated copy $\Om+x$ of $\Om$ can even partly cover $S$
unless this translated copy also covers a part of 
$\Om$ as well. The theorem says that a set $\Om$ satisfying 
this assumption cannot be spectral.

\begin{proof}[Proof of \thmref{thmA2}]
Suppose to the contrary that $\Om$ is spectral. Then by \thmref{thmA11}
the complement $\Om^\cm$ of $\Om$ admits a weak tiling by translates of $\Om$,
so  there is a positive, locally finite measure $\mu$ such that
$\1_\Om \ast \mu = \1_{\Om^\cm}$ a.e. 
Using condition \ref{thmA2.2},  we then have
\begin{equation}
\label{eq:mes_s}
m(S) = m(S \cap \Om^\cm) =  \int_S \1_{\Om^\cm} \, dm = \int_S (\1_\Om \ast \mu) \, dm
= \int  \varphi \, d\mu,
\end{equation}
where $\varphi(x) := m((\Om+x) \cap S)$. The measure
$\mu$ is supported on $\Delta(\Om)^\cm$ by
\corref{corA8.1}, while the function $\varphi$ vanishes 
on  $\Delta(\Om)^\cm$ due to condition \ref{thmA2.3}.
Hence the right hand side of 
\eqref{eq:mes_s} must be zero. It follows that $m(S) = 0$, a
contradiction to condition \ref{thmA2.1}.
\end{proof}

In Figure \ref{fig:img0942} we illustrate an example of a planar domain
that  can be shown to be non-spectral using \thmref{thmA2}. The
essential feature of this domain is that it is connected but its complement 
is not connected, so that there is a ``hole'' inside the domain.
We  observe that if $S$ is a set of positive measure contained
in the hole,  then S satisfies the conditions \ref{thmA2.1}, \ref{thmA2.2} 
and \ref{thmA2.3} in \thmref{thmA2}, and therefore the  domain
cannot be spectral.


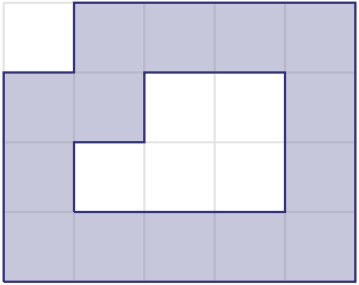
\begin{figure}[ht]
\centering

\begin{tikzpicture}[scale=0.925, style=mystyle]
\fill[myblue,opacity=0.275] 
	(1,1) -- (6,1) -- (6,5) -- (5,5) -- (2,5) -- 
	(2,4) -- (1,4) -- (1,3) -- (3,3) -- (3,4) -- 
	(5,4) -- (5,2) -- (2,2) -- (2,3) -- (1,3) -- (1,1);

\draw[gray,opacity=0.2] (1,1) grid (6,5);

\draw[myblue]
	(1,1) -- (6,1) -- (6,5) -- (5,5) -- 
	(2,5) -- (2,4) -- (1,4) -- (1,1);

\draw[myblue] 
	(2,2) -- (5,2) -- (5,4) -- 
	(3,4) -- (3,3) -- (2,3) -- (2,2);
\end{tikzpicture}

\caption{The illustration presents a planar domain that 
	can be shown to be non-spectral using \thmref{thmA2}.}
\label{fig:img0942}
 \end{figure}


As another example, we can use  \thmref{thmA2} to obtain the following result:

\begin{thm}
\label{thmB1}
Let $\Om$ be a bounded, open set  in $\R^d$. If $\Om$ is spectral,
then its boundary $\bd{\Om}$ must be a set of Lebesgue measure zero.
\end{thm}

\begin{proof}
Suppose that $\bd{\Om}$ has positive Lebesgue measure. 
We will show that the set $S := \bd{\Om}$
satisfies the conditions \ref{thmA2.1}, \ref{thmA2.2} and \ref{thmA2.3} 
in \thmref{thmA2}, and hence $\Om$ cannot be spectral.
 Indeed,  condition \ref{thmA2.1} holds by assumption.
Since $\Om$ is an open set, the two sets $\Om$ and $S$ are disjoint, hence
condition \ref{thmA2.2} holds as well. 
To verify condition \ref{thmA2.3}, we let $x \in \R^d$ be such that
 $m((\Om+x) \cap S)>0$. In particular, the set
 $(\Om+x) \cap S$ is nonempty. The set $\Om+x$ is open, while
 $S$ is contained in the closure of $\Om$, so it follows that
 $(\Om+x) \cap \Om$ must be nonempty. But the set $(\Om+x) \cap \Om$ is open, since it is 
the intersection of two open sets, so it can be nonempty only if
$m((\Om+x) \cap \Om) > 0$. This confirms
 condition \ref{thmA2.3}, and thus the proof is  concluded
by  \thmref{thmA2}. 
\end{proof}

On the other hand, if $\Om$ is a convex body in $\R^d$, then
there exists no set $S$ satisfying 
the conditions \ref{thmA2.1}, \ref{thmA2.2} and \ref{thmA2.3}
in \thmref{thmA2}. So in order to study  the spectrality problem 
for convex bodies, we must use the weak tiling condition
$\1_\Om \ast \mu = \1_{\Om^\cm}$ a.e.\ in a stronger way.
 This will be done  in the following sections.


\section{Spectral convex bodies are polytopes}
\label{sect:polytope}

In this section we  prove \thmref{thmA20}, which states that if
a convex body $\Om \subset \R^d$ is  spectral, then 
$\Om$ must  be  a convex polytope. This shows that
the first condition \ref{vm:i} in the Venkov-McMullen theorem
is  necessary  not only for tiling by translations, but also for the
spectrality of a convex body $\Om$.

The fact that a convex body which tiles by translations must be a  polytope,
is a classical result that is due to Minkowski, see \cite{McM80} 
or \cite[Section 32.2]{Gru07}. We will prove a stronger version of this result,
which involves only a weak tiling assumption:

\begin{thm}
\label{thmE5}
Let $K$ be a convex body in $\R^d$. If the complement $K^\cm$
of $K$ has a weak tiling by translates of $K$, then $K$ must be a 
convex polytope.
\end{thm}

Combining \thmref{thmE5} with \thmref{thmA11} yields that
any spectral convex body in $\R^d$ must be a convex polytope,
and \thmref{thmA20} thus follows.

The rest of the section is devoted to the proof of  \thmref{thmE5}.

\subsection{}
We start by recalling some basic facts about convex bodies in $\R^d$.
For more details, the reader is referred to \cite[Sections 4.1, 5.1, 14.1]{Gru07}.

Recall that a set $K \subset \R^d$ is called a \emph{convex body} 
if $K$ is a compact, convex set with nonempty interior. In what follows, we 
assume $K$  to be  a  convex body in $\R^d$.

A hyperplane $H$ is called a \emph{support hyperplane} of $K$ at 
a point $x \in \bd{K}$, if $x \in H$
and $K$ is contained in one of the two closed halfspaces 
bounded by $H$. In this case, we denote the
halfspace containing $K$ by $H^-$, and the other one by 
$H^+$. $H^-$ is called a \emph{support halfspace} of $K$ 
at the point $x$. It can be represented in the form
\[
H^- = \{z \in \R^d : \dotprod{z}{\xi} \leq \dotprod{x}{\xi} \},
\]
where $\xi$  is a vector in $\R^d$ such that
$|\xi|=1$. We call $\xi$ an \emph{exterior normal unit
vector} of $K$, or of $H$, at the  point $x$.

For each $x \in \bd{K}$, there exists at least one support hyperplane
$H$ of $K$ at $x$. The  support hyperplane at $x$ 
need not, in general, be unique. If the  support hyperplane 
 is unique then $x$ is called a \emph{regular} boundary point
of $K$, and otherwise it is called a \emph{singular} boundary 
point.

For each vector $\xi \in \R^d$, $|\xi|=1$, there is a unique
support hyperplane $H(K, \xi)$ of $K$ with exterior normal vector $\xi$.
The set $S(K, \xi) := K \cap H(K, \xi)$  is  called the \emph{support set}
of $K$ determined by $\xi$. The support set 
 is a compact, convex subset of $\bd{K}$.

A hyperplane $H$ is said to \emph{separate} two convex bodies
$K$ and $K'$, if  $K$ is contained in one of the two closed halfspaces 
bounded by $H$, while $K'$ is contained in the other closed halfspace.
Any two convex bodies $K$ and $K'$ with disjoint interiors have at
least one  separating hyperplane $H$.

A \emph{convex polytope}  $P \subset \R^d$ is the convex hull of a finite number
of points. Equivalently,  a convex polytope is a bounded set $P$ which
can be represented as the intersection
of finitely many closed halfspaces.

\subsection{}

\begin{lem}
\label{lemE1.1}
Let $K$ be a convex body in $\R^d$. If $K$ is not a polytope,
then there exists an infinite sequence $x_n$ of 
regular  boundary points of $K$ such that the corresponding
exterior normal unit vectors $\xi_n$ are distinct, 
that is, $\xi_n \neq \xi_m$ whenever $n \neq m$.
\end{lem}

\begin{proof}
We will rely on the fact that if $K \subset \R^d$ is a convex body,
then the set of regular boundary points of $K$ constitutes a dense
subset of $\bd{K}$ (see \cite[Section 5.1, Theorem 5.1 or Theorem 5.2]{Gru07}).

We construct the sequence $x_n$ by induction. Let $x_1$ be 
any regular  boundary point of $K$, and $\xi_1$ be the 
exterior normal unit vector at the point $x_1$. Now suppose
that the points $x_1, x_2, \dots, x_{n-1}$ have already been
chosen, that they are regular boundary points of $K$,
and their corresponding
exterior normal unit vectors $\xi_1, \xi_2, \dots, \xi_{n-1}$
 are distinct. 

Let
\[
F_n := \bigcup_{j=1}^{n-1} S(K, \xi_j),
\]
where $S(K, \xi_j)$ is the support set of $K$ determined by $\xi_j$. 
Then $F_n$ is a closed subset of $\bd{K}$. 
We claim that $F_n$ must be a proper subset
of $\bd{K}$. Indeed, suppose to the contrary that $\bd{K} = F_n$.
Then each point $y \in \bd{K}$ belongs to $S(K, \xi_{j(y)})$
for some $1 \leq j(y) \leq n-1$. Hence the closed halfspace
\[
H_y^- := \{z : \dotprod{z}{\xi_{j(y)}} \leq \dotprod{x_{j(y)}}{\xi_{j(y)}} \}
\]
is a support  halfspace of $K$ at the point $y$. It is known
(see \cite[Section 4.1, Corollary 4.1]{Gru07}) that if
for each   $y \in \bd{K}$, $H_y^-$ is a support  halfspace of $K$ at the point $y$,
then 
\[
K = \bigcap_{y \in \bd{K}} H_y^{-}.
\]
But the collection 
$\{ H_y^-  : y \in \bd{K} \}$ has at most $n-1$ distinct members.
It follows that $K$ is the intersection
of finitely many closed halfspaces,  hence $K$ is a convex polytope. Since
we have assumed that $K$ is not a polytope, we arrive at
a contradiction. This establishes our claim that $F_n$ must be a proper subset
of $\bd{K}$. 

Since $F_n$ is a closed set, it  follows 
that its complement $\bd{K} \setminus F_n$ is a nonempty, relatively
open subset of $\bd{K}$. The set of regular boundary points 
of $K$ is  dense in $\bd{K}$, hence we can choose a regular point
$x_n$ in $\bd{K} \setminus F_n$. The exterior normal unit vector
 $\xi_n$ at the point $x_n$ is then distinct from all the
vectors $\xi_1, \xi_2, \dots, \xi_{n-1}$. This completes the
inductive construction, and thus concludes the proof of the lemma.
\end{proof}

\subsection{}

\begin{lem}
\label{lemE2.1}
Let $K$ be a convex body in $\R^d$, let $x$ be a regular boundary 
point of $K$, and let $\xi$ be the exterior normal unit vector at the point $x$. 
Suppose that $U$ is an open neighborhood of the set $-S(K, -\xi) + x$,
where 
$S(K, -\xi)$ is the support set of $K$ determined by the vector $-\xi$.
Then there is an open neighborhood  $V$ of $x$ 
such that for any
 $t \in \Delta(K)^\cm$, the set $K+t$ cannot intersect
$V$ unless $t \in U$.
\end{lem}

\begin{proof}
Suppose to the contrary that the assertion is not true.
Then there is a sequence $t_j \in \Delta(K)^\cm \cap U^\cm$,
and for each $j$ there is a point $x_j \in K+t_j$, such that $x_j \to x$.
Since  $t_j$ must remain bounded as $j \to \infty$, we may
assume, by passing to a subsequence if needed, that  $t_j \to t$.

Since the set $\Delta(K)^\cm$ is closed, the limit $t$ of the sequence
$t_j$ belongs to $\Delta(K)^\cm$. Hence $K$ and $K + t$ are two 
convex bodies with disjoint interiors. It follows that there exists a
hyperplane $H$ separating $K$ from $K+t$.
The point $x$ lies on both $K$ and $K+t$, and therefore $x$ must
lie on $H$. It follows that $x$ is a 
boundary point of both $K$ and $K+t$, and 
 $H$ is a support hyperplane of both $K$ and $K+t$ at the point $x$.
Since $x$ is a regular boundary point of $K$,
$H$ is the unique support hyperplane $H(K, \xi)$ of $K$ at the point $x$.

Now  $H(K, \xi)$ is a support hyperplane of $K+t$ at the point $x$,
 with exterior normal vector $-\xi$. It follows that $H(K, \xi) - t$ is a support 
hyperplane of $K$ at the point $x-t$, with 
exterior normal vector $-\xi$. Since a support hyperplane of $K$ with 
a given exterior normal vector is unique, we must have that
$H(K,\xi) - t = H(K, -\xi)$, and hence the point $x-t$ lies on the support
set $S(K, -\xi)$ determined by the vector $-\xi$. 
In other words, we have obtained that $t \in  - S(K, -\xi) + x$. 
But since the latter set is contained
in $U$, we conclude that $t \in U$. However this is not possible, since
$t$ is the limit of a sequence $\{t_j\} \subset U^\cm$ and the set $U$ is open.
 This contradiction completes the proof.
\end{proof}

\subsection{}

\begin{lem}
\label{lemE1.2}
Let $K \subset \R^d$ be a convex body  whose complement
$K^\cm$  admits a weak tiling by translates of $K$, that is,
there is a positive, locally finite measure $\mu$ such that
$\1_K \ast \mu = \1_{K^\cm}$ a.e. 
Let $x$ be a regular boundary point of $K$, and let $\xi$ be the 
exterior normal unit vector at the point $x$.
Then we have \[ \mu(- S(K,-\xi) + x ) \geq 1, \] where
$S(K, -\xi)$ is the support set
of $K$ determined by the vector $-\xi$. 
\end{lem}

\begin{proof}
Suppose to the contrary that  $\mu(- S(K,-\xi) + x ) < 1$.
Then there is an open neighborhood $U$ of the set
$- S(K,-\xi) + x$ such that also $\mu(U) < 1$. By \lemref{lemE2.1} 
there is an open neighborhood $V$ of $x$ such that for any
 $t \in \Delta(K)^\cm$, the set $K+t$ cannot intersect
$V$ unless $t \in U$.

We now decompose the measure
$\mu$ into a sum $\mu = \mu' + \mu''$, where
$\mu' := \mu \cdot \1_U$ and
$\mu'' := \mu \cdot \1_{U^\cm}$.
The support of the measure $\mu$ must be contained in
$\Delta(K)^\cm$ due to the weak tiling assumption
(\corref{corA8.1}), and therefore $\mu''$ is supported on
the closed set $\Delta(K)^\cm \cap U^\cm$. Hence
for any $t \in \supp(\mu'')$, the set $K+t$ does not intersect
$V$, which implies that $\1_K \ast \mu'' = 0$ a.e.\ in $V$.
It follows that
\[
\1_{V \cap K^\cm} = 
\1_{V} \cdot (\1_K \ast \mu ) = 
\1_{V} \cdot (\1_K \ast \mu' ) + \1_{V} \cdot (\1_K \ast \mu'' ) =
\1_{V} \cdot (\1_K \ast \mu' )
\quad \text{a.e.,}
\]
and therefore
\[
\| \1_{V \cap K^\cm} \|_{L^\infty(\R^d)} \leq 
\| \1_K \ast \mu'  \|_{L^\infty(\R^d)} \leq \int d\mu' = \mu(U) < 1.
\]
This implies that the set $V \cap K^\cm$ must have Lebesgue measure zero.
But this is not possible, since $V$ is an open neighborhood of
the point $x \in \bd{K}$, and hence $V \cap K^\cm$ has  nonempty interior.
 We thus arrive at a contradiction, which concludes the proof.
\end{proof}

\subsection{}
\begin{lem}
\label{lemE1.3}
Let $K$ be a convex body in $\R^d$, and 
$x, x'$ be two regular boundary points of $K$. 
Let $\xi$ and $\xi'$ be the exterior normal unit vectors
at the points $x$ and $x'$ respectively, and let
$S(K, -\xi)$ and $S(K, -\xi')$ be the support sets
of $K$ determined by $-\xi$ and $-\xi'$
respectively. If we have $\xi \neq \xi'$, then the two sets
$- S(K,-\xi) + x$ and $- S(K,-\xi') + x'$  are disjoint.
\end{lem}

\begin{proof}
Suppose to the contrary that the two sets
$- S(K,-\xi) + x$ and $- S(K,-\xi') + x'$  are not disjoint, so they
have at least one point in common.
Then there exist two points $y \in S(K,-\xi)$ and $y' \in S(K,-\xi')$
such that $x - y = x' - y'$.

The support hyperplanes of $K$ with exterior normal vectors $\xi$ and $-\xi$
are respectively given by
\begin{equation}
\label{eq:E1.3.1}
H(K,\xi) = \{z  : \dotprod{z}{\xi} = \dotprod{x}{\xi} \}
\quad \text{and} \quad
H(K,-\xi) = \{z  : \dotprod{z}{\xi} = \dotprod{y}{\xi} \}.
\end{equation}
Since $K$ is contained in the closed slab bounded
by these two hyperplanes, we have
\begin{equation}
\label{eq:E1.3.2}
K \subset \{z  :  \dotprod{y}{\xi} \leq  \dotprod{z}{\xi} \leq \dotprod{x}{\xi} \}.
\end{equation}
The point $x'$ is  a regular boundary point of $K$, and so $H(K, \xi')$ 
is the unique support hyperplane of $K$ at the point
$x'$. Since $\xi \neq \xi'$, it follows that $x'$
does not lie on $H(K,\xi)$. Using \eqref{eq:E1.3.1} and \eqref{eq:E1.3.2} 
this implies that $\dotprod{x'}{\xi} < \dotprod{x}{\xi}$.

Now denote $h := x - y = x' - y'$, then it follows that
\[
\dotprod{y'}{\xi} = \dotprod{x' - h}{\xi} <
\dotprod{x - h}{\xi} = \dotprod{y}{\xi}.
\]
But since $y' \in K$ this contradicts \eqref{eq:E1.3.2},
and thus the lemma is proved.
\end{proof}

\subsection{}
\begin{proof}[Proof of \thmref{thmE5}]
Assume that $K \subset \R^d$ is a convex body   whose
complement $K^\cm$ has a weak tiling by translates of $K$.
We will prove that $K$ is necessarily a polytope.

Suppose to the contrary that  $K$ is not a polytope. Then
by \lemref{lemE1.1} there exists an infinite sequence $x_n$ of 
regular boundary points of $K$, such that the corresponding
exterior normal unit vectors $\xi_n$ are mutually distinct.

The set $K^\cm$  has a weak tiling by translates of $K$,   so
there is a positive, locally finite measure $\mu$ such that
$\1_K \ast \mu = \1_{K^\cm}$ a.e.
Let $S(K, -\xi_n)$ be the support set
of $K$ determined by the vector $-\xi_n$, then by \lemref{lemE1.2} 
we have 
$\mu(- S(K, -\xi_n) + x_n) \geq 1$ for each $n$.

Since the exterior normal unit vectors $\xi_n$ are distinct,
we get from \lemref{lemE1.3} that the sets $- S(K, -\xi_n) + x_n$
are pairwise disjoint. On the other hand, all these sets
are contained in  $K-K$.
Hence the set $K-K$ contains an infinite sequence 
of pairwise disjoint subsets, such that the total mass of $\mu$
in each one of these sets is at least $1$.
It follows that we must have $\mu(K-K) = +\infty$. But  this is not possible, as
$K-K$ is a bounded set and $\mu$ is a locally finite measure. 
We thus  arrive at a contradiction, which concludes the proof
of the theorem.
\end{proof}


\section{Spectral convex polytopes can tile by translations, I}
\label{sect:belti}

So far, we have established that a spectral
convex body $\Om \subset \R^d$ must be  a convex polytope
(\thmref{thmA20}). We also know from the result in
\cite{Kol00} (or \cite{KP02}) that $\Om$ must be centrally
symmetric, and from the result in \cite[Section 4]{GL17} 
that all the facets of $\Om$ must be centrally symmetric as well. 
It now remains to prove that each belt of $\Om$ must consist of 
either $4$ or $6$ facets (\thmref{thmA21}). 
The proof will be given throughout the present
section and the next one.

The key results of the present section are 
\lemref{lemJ2.7} and \lemref{lemJ2.10}.

\subsection{}
We begin by recalling some basic facts about convex polytopes in $\R^d$.
For more details, we refer  to \cite[Section 1.A]{BG09},
\cite[Section 14.1]{Gru07} and \cite[Section 2.4]{Sch14}.

A \emph{convex polytope}  $A \subset \R^d$ is the convex hull of a finite number
of points. Equivalently,  a convex polytope is a bounded set $A$ which
can be represented as the intersection
of finitely many closed halfspaces. 

We denote by $\aff(A)$ the \emph{affine hull} of 
a convex polytope $A \subset \R^d$,  that is, $\aff(A)$ is the smallest
affine subspace containing $A$. By the \emph{relative interior} and 
\emph{relative boundary}  of $A$ we refer respectively to the
interior and boundary relative to $\aff(A)$.  The relative interior
of $A$ will be denoted by $\relint(A)$.

A \emph{face} of  a convex polytope $A$ is  
a support set of $A$, that is,  the intersection of $A$ with 
a support hyperplane  of $A$. (There is also an alternative, equivalent definition,
according to which a face of $A$ is an \emph{extreme subset}
of $A$, that is, a convex subset $F \subset A$ such that if $x,y \in A$, 
$(1-\lam)x + \lam y \in F$, $0 < \lam < 1$, then $x,y \in F$.)

If $A \subset \R^d$ is a convex polytope with nonempty interior, then
a $(d-1)$-dimensional face of $A$ is called a \emph{facet} of $A$,
while a $(d-2)$-dimensional face is called a \emph{subfacet} of $A$.
If $G$ is a subfacet of $A$, then there exist exactly two facets $F_1$
and $F_2$ of $A$ which contain $G$, and we have $G = F_1 \cap F_2$
(see \cite[Proposition 1.12]{BG09}).

A convex polytope $A \subset \R^d$ is said to be
 \emph{centrally symmetric} if the set $-A$ is a translate
of $A$. In this case, there is a unique point $x \in \R^d$ such that 
$- A + x = A - x$. The point $x$ is called the \emph{center of symmetry} of $A$, 
and $A$ is then said to be symmetric with respect to $x$.

If $X, Y$ are two convex subsets of $\R^d$, then we use
  $\convex \{X, Y\}$ to denote the convex hull of the union $X \cup Y$.

A \emph{prism} in $\R^d$ is a convex polytope of the form
$\convex \{F, F + \tau\}$, where $F$ is a  $(d-1)$-dimensional 
convex polytope, and $\tau$ is a vector such that 
 $F$ and $F+\tau$ do not lie on the same hyperplane.
The two sets 
 $F$ and $F+\tau$ are called the \emph{bases} of the prism.

A \emph{slab} in $\R^d$ is  the closed region between 
two parallel hyperplanes, that is, a set of the form
$ \{z : c_1 \leq \dotprod{z}{\xi} \leq c_2\}$, where
$\xi$ is a nonzero vector and $c_1, c_2$ are constants.

\subsection{}

\begin{lem}
\label{lemJ2.11}
Let $A \subset \R^d$ be a convex polytope with nonempty interior. Then
\[
\interior(A) - \interior(A) = \interior(A) - A = \interior(A-A).
\]
\end{lem}

\begin{proof}
The fact that $\interior(A)$ is a subset of $A$ implies that
$\interior(A) - \interior(A) \subset \interior(A) - A$.
In turn, $\interior(A) - A$ is the union of all  sets
of the form $\interior(A) - x$ where $x \in A$, and
all these sets are open.
Hence $\interior(A) - A$ is an open subset of $A-A$, and it follows
that $\interior(A) - A \subset \interior(A-A)$. It remains to prove
that $\interior(A-A) \subset \interior(A) - \interior(A)$. Indeed,
 let $h \in \interior(A-A)$.
Then there is $\eps  > 0$ such that $h + \eps h$ is in $A-A$.
Let $x, y \in A$ be such that $h + \eps h = y - x$. Let $a$
be an interior point of $A$, and define the points
$x' := \lambda x + (1-\lambda)a$ and
$y' := \lambda y + (1-\lambda)a$,
where $\lam := (1+\eps)^{-1}$. Then $x', y' \in \interior(A)$ and we have
$y'-x' = h$, which shows that $h \in \interior(A) - \interior(A)$.
\end{proof}

\begin{lem}
\label{lemJ2.14}
Let $A \subset \R^d$ be a convex polytope with nonempty
interior, and let $F$ be a facet of $A$. Then
\[
\relint(F-F) \subset \interior(A-A).
\]
\end{lem}

\begin{proof}
Let $h \in \relint(F-F)$.
Then there is an open set $U$  such that
$h \in U \cap H \subset F-F$, where  $H$ denotes the hyperplane through 
the origin parallel to $F$.
It follows that if we choose $\eps > 0$ small
enough then $h + \eps h$ is in $F-F$.
Let $x, y \in F$ be such that $h + \eps h = y - x$. Let $a$
be an interior point of $A$, and define the points
$x' = \lambda x + (1-\lambda)a$ and
$y' = \lambda y + (1-\lambda)a$,
where $\lam = (1+\eps)^{-1}$. Then $x', y' \in \interior(A)$ and we have
$y'-x' = h$, which shows that $h \in \interior(A) - \interior(A)$.
Using \lemref{lemJ2.11} we conclude that
$h \in \interior(A-A)$ and the lemma is proved.
\end{proof}

\begin{lem}
\label{lemJ2.5}
Let $A \subset \R^d$ be a convex polytope with nonempty
interior. Let $F$ be a facet of $A$, and let $H^-$
be the support halfspace of $A$ at  the facet $F$. Suppose
that $E$ is a compact subset of $\relint(F)$. Then
there is an open neighborhood $V$ of $E$ such that the set
$V \cap H^{-}$ is contained in $A$.
\end{lem}

\begin{proof}
Let $F_i$, $i=1,2,\dots,m$, be all the facets of $A$, and
suppose that $F = F_1$. For each $i$, let $H_i$ be the 
hyperplane containing the facet $F_i$, and let $H^{-}_i$ 
be the support halfspace of $A$ at the facet $F_i$. Then
\begin{equation}
\label{eq:J2.5.1}
A= H^{-}_1 \cap H^{-}_2 \cap \dots \cap H^{-}_m
\end{equation}
(see \cite[Theorem 14.2]{Gru07}). In particular, the set $E$ is contained
in all the halfspaces $H^{-}_i$. On the other hand no point of $E$ can
lie on any one of the hyperplanes $H_i$, $i \neq 1$, since $E$ is a
subset of $\relint(F)$. Hence there is 
an open neighborhood $V$ of $E$ such that 
$V$ is contained in $H^{-}_i$ for all $i \neq 1$.
Using this together with \eqref{eq:J2.5.1} implies the claim.
\end{proof}

\begin{lem}
\label{lemJ2.6}
Let $A \subset \R^d$ be a convex polytope with nonempty
interior, and let $G$ be a subfacet of $A$. Let
$F_1$ and $F_2$ be the two adjacent facets of $A$ that meet at
the subfacet $G$, and let $H^{-}_1$ and $H^{-}_2$ be the
support halfspaces of $A$ at the facets $F_1$ and $F_2$
respectively. Suppose that  $E$ is a compact subset of $\relint(G)$. Then
there is an open neighborhood $V$ of $E$ such that the set
$V \cap H^{-}_1 \cap H^{-}_2$ is contained in $A$.
\end{lem}

\begin{proof}
We continue to use the same notations as in the proof of Lemma \ref{lemJ2.5}. The set $E$ is again contained
in all the halfspaces $H^{-}_i$. However no point of $E$ can
lie on any one of the hyperplanes $H_i$, $i \neq 1,2$, since 
$E$ is a subset of $\relint(G)$. Hence there is 
an open neighborhood $V$ of $E$ such that 
$V$ is contained in $H^{-}_i$ for all $i \neq 1,2$.
Using this with \eqref{eq:J2.5.1} we obtain the
assertion of the lemma.
\end{proof}

\subsection{}

\begin{lem}
\label{lemJ2.15}
Let $A \subset \R^d$ be a convex polytope with nonempty
interior, and let $F$ be a facet of $A$. Suppose that
$\xi$ is the exterior normal unit vector
of $A$ at the facet $F$, and that
 $H^- = \{z : \dotprod{z}{\xi} \leq c\}$ is the support
halfspace of $A$ at the facet $F$. For each $\delta > 0$ we let
$P_\delta = A \cap S_\delta$
be the intersection of $A$ with the slab
$S_\delta = \{z : c - \delta \leq \dotprod{z}{\xi} \leq c\}$, and
 we let $Q_\delta = \convex\{F, F - \delta \xi\}$ 
be the prism with bases $F$ and $F - \delta \xi$. Then we have
$m(P_\delta \triangle Q_\delta) = o ( \delta )$ as $\delta \to 0$
(see Figure \ref{fig:img0924}).
\end{lem}


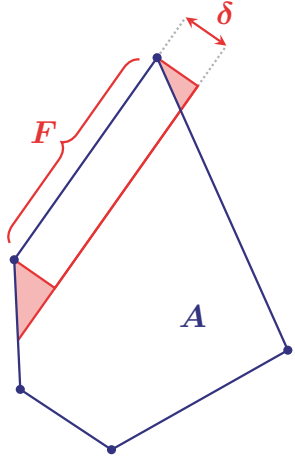
\begin{figure}[ht]
\centering

\begin{tikzpicture}[scale=0.04, style=mystyle]

\def\ax{95};
\def\ay{102};
\def\bx{138};
\def\by{5};
\def\cx{80};
\def\cy{-28}
\def\dx{50};
\def\dy{-8};
\def\ex{48};
\def\ey{35};
\def\aax{{\ax+0.2*(\ay-\ey)}};
\def\aay{{\ay+0.2*(\ex-\ax)}};
\def\eex{{\ex+0.2*(\ay-\ey)}};
\def\eey{{\ey+0.2*(\ex-\ax)}};

\def\hx{1.2 * \ax - 0.2 * \ex};
\def\hy{1.2 * \ay - 0.2 * \ey};
\def\jx{{(\hx)+0.2*(\ay-\ey)}};
\def\jy{{(\hy)+0.2*(\ex-\ax)}};

\coordinate (A) at (\ax,\ay);
\coordinate (B) at (\bx,\by);
\coordinate (C) at (\cx,\cy);
\coordinate (D) at (\dx,\dy);
\coordinate (E) at (\ex,\ey);
\coordinate (AA) at (\aax,\aay);
\coordinate (EE) at (\eex,\eey);
\coordinate (H) at (\hx,\hy);
\coordinate (J) at (\jx,\jy);

\draw [gray, opacity=0.6, densely dotted] (A) -- (H);
\draw [gray, opacity=0.6, densely dotted] (AA) -- (J);

\fill [myred, opacity=0.35]
	(A) -- (AA) -- (intersection of AA--EE and A--B) -- (A);

\fill [myred, opacity=0.35]
	(E) -- (EE) -- (intersection of AA--EE and D--E) -- (E);

\draw [myred] (A) -- (AA) -- (EE) -- (E);
\draw [myred] (intersection of AA--EE and D--E) -- (AA) ;

\draw (A) -- (B) -- (C) -- (D) -- (E) -- (A);

\fill (A) circle (\mycirc);
\fill (B) circle (\mycirc);
\fill (C) circle (\mycirc);
\fill (D) circle (\mycirc);
\fill (E) circle (\mycirc);

\draw (B) 
	node[xshift=-1.25cm,yshift=0.45cm] 
	{$\boldsymbol{A}$};

\draw [decorate,decoration={brace,raise=5pt,amplitude=6pt,
	pre=moveto, pre length=0.15cm,
	post=moveto, post length=0.15cm}, myred]
	(E) -- (A) 
	node [myred,midway,xshift=-0.55cm,yshift=0.35cm] 
	{$\boldsymbol{F}$};

\draw [<->, >=stealth, myred] (H) -- (J)
	node [myred,midway,xshift=0.25cm,yshift=0.25cm] 
	{$\boldsymbol{\delta}$};

\end{tikzpicture}

\caption{The shaded region in the illustration represents the symmetric
	difference $P_\delta \triangle Q_\delta$ 
	of the sets $P_\delta$ and $Q_\delta$ in \lemref{lemJ2.15}.}
\label{fig:img0924}
\end{figure}


\begin{proof}
Let $H$ be the hyperplane containing the facet $F$.
For each $\eps > 0$, let $F_{+\eps}$ be the set of all
points of $H$ whose distance from $F$ is at most $\eps$,
and $F_{-\eps}$ be the set of all
points of $F$ whose distance from the relative boundary of $F$ is at least $\eps$.

Let $I_\delta = \{s  \xi : 0 \leq s \leq \delta\}$
 be the closed line segment connecting the origin
to the point $\delta \xi$. Then there is  $\delta_0 = \delta_0(A,F,\eps) > 0$
such that for every $\delta < \delta_0$ we have
\[
F_{-\eps} - I_\delta \subset P_\delta \subset F_{+\eps}  - I_\delta.
\]
(the left inclusion may be deduced from \lemref{lemJ2.5}).
Observe that  $Q_\delta = F - I_\delta$, hence
\[
P_\delta \triangle Q_\delta \subset
(F_{+\eps} \setminus F_{-\eps}) - I_\delta.
\]
This implies that
\[
m_d (P_\delta \triangle Q_\delta) \leq \delta \, m_{d-1}
 (F_{+\eps} \setminus F_{-\eps}).
\]
Since $m_{d-1}
 (F_{+\eps} \setminus F_{-\eps})$ tends to zero as $\eps \to 0$,
the assertion follows.
\end{proof}

\subsection{}

\begin{lem}
\label{lemJ2.13}
Let $A \subset \R^d$ be a convex polytope with nonempty
interior, and let $G$ be a subfacet of $A$. Let
$F_1$ and $F_2$ be the two adjacent facets of $A$ that meet at
the subfacet $G$, and let $H^{-}_1$ and $H^{-}_2$ be the
support halfspaces of $A$ at the facets $F_1$ and $F_2$
respectively. For each $\delta > 0$ we let 
$P_\delta$ be the set of all points of $A$ whose distance from
$\aff(G)$ is not greater than $\delta$, and we let
$Q_\delta =  (G + S_\delta) \cap H^{-}_1 \cap H^{-}_2$,
where $S_\delta$ is a closed $2$-dimensional ball of radius $\delta$  centered at the
origin and orthogonal to $\aff(G)$. Then we have
$m(P_\delta \triangle Q_\delta) = o ( \delta^2 )$ as $\delta \to 0$.
\end{lem}

\begin{proof}
The proof is similar to that of \lemref{lemJ2.15}.
For each $\eps > 0$, let $G_{+\eps}$ be the set of all
points of $\aff(G)$ whose distance from $G$ is at most $\eps$,
and $G_{-\eps}$ be the set of all
points of $G$ whose distance from the relative boundary of $G$ is at least $\eps$.
If $\delta_0 = \delta_0(A,G,\eps) > 0$ is small enough, then for
every $\delta < \delta_0$ we have
\[
(G_{-\eps} + S_\delta)
\cap H^{-}_1 \cap H^{-}_2
\subset
P_\delta \subset  (G_{+\eps} + S_\delta)
\cap H^{-}_1 \cap H^{-}_2
\]
(here, the left inclusion may be inferred from \lemref{lemJ2.6}),
and hence
\[
P_\delta \triangle Q_\delta \subset
((G_{+\eps} \setminus G_{-\eps}) + S_\delta)
\cap H^{-}_1 \cap H^{-}_2.
\]
This implies that
\[
m_d (P_\delta \triangle Q_\delta) \leq 
m_{d-2} (G_{+\eps} \setminus G_{-\eps})
\, m_2( S_\delta).
\]
But we have
$ m_2( S_\delta) = \pi \delta^2$, while
$m_{d-2} (G_{+\eps} \setminus G_{-\eps})$
 tends to zero as $\eps \to 0$.
\end{proof}

\subsection{}

\begin{lem}
\label{lemJ2.8}
Let $A$ and $B$ be two convex polytopes in $\R^d$ with nonempty interiors. 
Let $L$ be one of the facets of $B$, let $\xi$ 
be the exterior normal unit vector of $B$ at the facet $L$,
and let $E$ be a compact subset of $\relint(L)$.
Suppose that $A$ has a facet $F$ on which
the exterior normal unit vector is $- \xi$,
and let $U$ be an open neighborhood of the set $E - F$. Then
there is an open neighborhood  $V$ of $E$ 
such that for any $t$, if $A+t$ and $B$ have disjoint
interiors and if $A+t$ intersects $V$, then $t \in U$.
\end{lem}

This can be proved in essentially the same way as \lemref{lemE2.1} above.

\subsection{}

\begin{lem}
\label{lemJ2.7}
Let $A$ and $B$ be two  convex polytopes in $\R^d$
 with nonempty, disjoint interiors, and suppose that $A$ and $B$ share
a common facet $F$. Assume that
$\mu$ is a positive, locally finite measure such that
$\1_A \ast \mu \geq 1$ a.e.\ on $A$, while
$\1_A \ast \mu = 0$ a.e.\ on $B$. 
Let $\mu'$ denote the restriction of the measure $\mu$ to
$\relint(F - F)$. Then we have
$\1_{F} \ast \mu' \geq 1$ a.e.\ with respect to
the $(d-1)$-dimensional volume measure on the facet $F$.
\end{lem}

Notice that the convolution $\1_{F} \ast \mu'$ vanishes
outside of $\aff(F)$, since $F+t$ lies on $\aff(F)$ for
every $t \in \supp(\mu')$. 
An equivalent way to  formulate the conclusion of the lemma 
is  to say that if $\sigma_F$ denotes
 the $(d-1)$-dimensional volume measure restricted to
 the facet $F$,
then we have $(\sigma_F \ast \mu')(E) \geq \sigma_F(E)$
for every Borel set $E$.

\begin{proof}[Proof of \lemref{lemJ2.7}]
By applying a rotation and a translation we may
assume that the facet $F$ is contained in the hyperplane
$\{ x: x_1 = 0 \}$. Hence 
$F$ has the form $F = \{ 0 \} \times \Om$,
where $\Om$ is a convex polytope in $\R^{d-1}$ with
nonempty interior.
We may also suppose that $A \subset \{ x: x_1 \geq 0 \}$
and  $B \subset \{ x: x_1 \leq 0 \}$.

Let $\eta > 0$, and let
$\Sig$ be a compact subset of $\interior(\Om)$.
Then the set $E := \{ 0 \} \times \Sig$
is a compact subset of $\relint(F) = \{  0 \} \times \interior(\Om)$.
Choose  $\eps>0$ such that the total mass of the measure $\mu$
in the set $((-\eps,0) \cup (0, \eps)) \times 
 \interior (\Om-\Om)$ is less than $\eta$, and define
 $U := (-\eps, \eps) \times  \interior (\Om-\Om)$.
Using \lemref{lemJ2.11} we  see that
$U$ is an open neighborhood of the set $E - F$.
By \lemref{lemJ2.8} there is an open neighborhood $V$ 
of $E$  such that for any
 $t$, if $A+t$ and $B$ have disjoint interiors
and if $A+t$ intersects
$V$, then $t \in U$.

Consider the following three subsets of $\R^d$:
\begin{enumerate-roman}
\item \label{J1.2.i} $Y' := \{0\} \times \interior (\Om - \Om) =  \relint(F-F)$;
\item \label{J1.2.ii} $Y'' := ((-\eps,0) \cup (0, \eps)) \times  \interior (\Om-\Om)$;
\item \label{J1.2.iii} $Y''' := (Y' \cup Y'')^\cm = U^\cm$.
\end{enumerate-roman}
Then the  sets $ Y', Y'', Y'''$ 
are pairwise disjoint, and they cover the whole space.
It follows that we may decompose the measure $\mu$ into
the sum $\mu = \mu' + \mu'' + \mu'''$, where the three measures
$\mu' , \mu'' , \mu'''$ are the restrictions of $\mu$ to
the sets $ Y', Y'', Y'''$ respectively.

The assumption that $\1_A \ast \mu = 0$ a.e.\ on $B$ implies that
the sets $A+t$ and $B$ must have disjoint interiors for every
$t \in \supp(\mu)$ (\lemref{lemA8.2}).
 Since the support of the measure $\mu'''$ is 
contained in $\supp(\mu) \cap U^\cm$, it follows that
if $t \in \supp(\mu''')$ then $A+t$ cannot intersect $V$.
 Hence we have $\1_A \ast \mu''' = 0$ a.e.\ in $V$.
We also have
\[
\| \1_A \ast \mu''  \|_{L^\infty(\R^d)} \leq 
\int_{\R^d} d\mu'' = \mu(Y'') < \eta.
\]
Combining this with
the assumption that $\1_A \ast \mu \geq 1$ a.e.\ on $A$,
 this implies that
\begin{equation}
\label{eq:J2.7.3}
\1_A \ast \mu'  =
\1_A \ast \mu  - \1_A \ast \mu'' -
\1_A \ast \mu''' \geq 1 - \eta \quad \text{a.e.\ on  $A \cap V$.}
\end{equation}

For each $\delta > 0$ we let
$P_\delta := A \cap S_\delta$ be the intersection of $A$ with
the slab $S_\delta :=  [0, \delta] \times \R^{d-1}$, and
we also consider the prism $Q_{\delta} := [0,  \delta] \times \Om$.
Then by \lemref{lemJ2.15}, we can choose $\delta$
 small enough such that
\begin{equation}
\label{eq:J2.7.16}
m(P_\delta \triangle Q_\delta) \, \mu(F-F) < \delta  \eta^2 \, m_{d-1}(\Sig).
\end{equation}
We can also assume, by choosing $\delta$ small enough,
that the set $D_\delta := [0,\delta] \times \Sigma$
is contained in both  $V$ and $A$
(the inclusion in $A$ can be deduced from \lemref{lemJ2.5}).

The support of the measure $\mu'$ is contained in 
the hyperplane $\{ 0 \} \times \R^{d-1}$. 
For each $t \in \supp(\mu')$ we therefore have 
$(A+t) \cap S_\delta = P_\delta + t$. This implies that
$ \1_A \ast \mu' = \1_{P_\delta} \ast \mu'$ a.e.\ on the
slab $S_\delta$. In particular,  it follows from \eqref{eq:J2.7.3} that
\begin{equation}
\label{eq:J2.7.5}
\1_{P_\delta} \ast \mu'  \geq 1 - \eta \quad \text{a.e.\ on  $D_\delta$.}
\end{equation}
Let $D'_\delta$ be the set of all points
$x \in D_\delta$ such that $(\1_{P_\delta \triangle Q_{\delta}} \ast \mu')(x)
  < \eta$, then we have
\[
\1_{Q_{\delta}} \ast \mu' \geq
\1_{P_\delta} \ast \mu' -
 \1_{P_\delta \triangle Q_{\delta}} \ast \mu'
\geq 1 - 2 \eta \quad \text{a.e.\ on $D'_\delta$,}
\]
which follows from \eqref{eq:J2.7.5}. On the other hand,
by \eqref{eq:J2.7.16} we have
\[
\int_{\R^d} (\1_{P_\delta \triangle Q_{\delta}} \ast \mu' )\, dm =
m(P_\delta \triangle Q_{\delta}) \, \int_{\R^d} d\mu'
< \delta \eta^2 \, m_{d-1}(\Sig),
\]
which in turn implies that
\[
m (D_\delta \setminus D'_\delta) 
\leq \eta^{-1}
\int_{\R^d} (\1_{P_\delta \triangle Q_{\delta}} \ast \mu' )\, dm 
< \delta \eta \, m_{d-1}(\Sig) = \eta \, m (D_\delta),
\]
that is, we have $m(D'_\delta) > (1-\eta) m(D_\delta)$.
We conclude that
\begin{equation}
\label{eq:J2.7.9}
\text{$\1_{Q_{\delta}} \ast \mu' \geq 1 - 2 \eta$\,
a.e.\ on $D'_\delta$}, \quad
D'_\delta \subset D_\delta, \quad
m(D'_\delta) > (1-\eta) m(D_\delta).
\end{equation}

Now recall that the support of the measure $\mu'$ is contained
in  $\{ 0 \} \times \R^{d-1}$. This
implies that the value of 
 $(\1_{Q_\delta} \ast \mu')(x)$ does not depend,
 in the slab $S_\delta$, on the first coordinate
$x_1$ of the point $x$.
Hence it follows from \eqref{eq:J2.7.9} that
$\1_{F} \ast \mu' \geq 1 - 2 \eta$ a.e.\ with
respect to the $(d-1)$-dimensional volume measure 
on some set of the form $\{0\} \times \Sig'$, where
$\Sig'$ is a subset of $\Sig$ (which depends on both
$\Sig$ and $\eta$) such that
$m_{d-1}(\Sig') \geq  (1-\eta) m_{d-1}(\Sig)$.

However, as the number $\eta$ was arbitrary, it follows
that   we must have
$\1_{F} \ast \mu' \geq 1$ a.e.\ with
respect to the $(d-1)$-dimensional volume measure 
on $\{0\} \times \Sig$. In turn, $\Sig$ was an arbitrary
compact subset of $\interior(\Om)$. Using the fact that
$\bd{\Om}$ is a set of measure zero in $\R^{d-1}$, this implies that
$\1_{F} \ast \mu' \geq 1$ a.e.\ with
respect to the $(d-1)$-dimensional volume measure 
on $\{0\} \times \Om = F$, which is what  we had to prove.
\end{proof}

\subsection{}

\begin{lem}
\label{lemJ2.10}
Let $A$ be a convex polytope in $\R^d$
 with nonempty interior, and let $G$ be a subfacet of $A$.
Let $F_1$ and $F_2$ be the two adjacent facets of $A$ that meet at
the subfacet $G$, and let $H^{-}_1$ and $H^{-}_2$ be the
support halfspaces of $A$ at the facets $F_1$ and $F_2$ respectively. 
For each $\delta > 0$ we let 
$Q_\delta :=  (G + S_\delta)
\cap H^{-}_1 \cap H^{-}_2$,
where $S_\delta$ is a closed $2$-dimensional ball of radius $\delta$  centered at the
origin and orthogonal to $\aff(G)$. 
Assume that $\mu$ is a positive, finite measure
supported on $G-G$ and satisfying
$\1_{G} \ast \mu \geq 1$ a.e.\ with respect to
the $(d-2)$-dimensional volume measure on  $G$.
Then for any $\eta > 0$ we have
\[
m \{ x  \in Q_\delta : (\1_{A} \ast \mu ) (x) < 1-\eta \} = o( m(Q_\delta) ),
\quad \delta \to 0.
\]
\end{lem}

\begin{proof}
Let $P_\delta = A \cap (\aff(G) + S_\delta)$ be the set of all points 
of $A$ whose distance from $\aff(G)$ is not greater than $\delta$.
The assumption that  $\supp(\mu) \subset G-G$ implies that 
we have $(A+t) \cap  (\aff(G) + S_\delta) = P_\delta + t$
for every $t \in \supp(\mu)$, and hence 
$\1_A \ast \mu = \1_{P_\delta} \ast \mu$
a.e.\ on  $\aff(G) + S_\delta$.  In particular,
\begin{equation}
\label{eq:J2.10.5}
\1_A \ast \mu = \1_{P_\delta} \ast \mu \quad \text{a.e.\ on $Q_\delta$.}
\end{equation}

We have also assumed
that $\1_{G} \ast \mu \geq 1$ a.e.\ with respect to
the $(d-2)$-dimensional volume measure on  $G$.
Using Fubini's theorem this implies that
\begin{equation}
\label{eq:J2.10.10}
\1_{Q_\delta} \ast \mu \geq 1 \quad \text{a.e.\ on $Q_\delta$.}
\end{equation}

Denote by $D_{\delta,\eta}$ the set of  all points
$x \in Q_\delta$ such that $(\1_{P_\delta \triangle Q_{\delta}} \ast \mu)(x)
  < \eta$. Then
\[
\1_A \ast \mu = \1_{P_\delta} \ast \mu \geq
\1_{Q_{\delta}} \ast \mu -
 \1_{P_\delta \triangle Q_{\delta}} \ast \mu
> 1 -  \eta \quad \text{a.e.\ on $D_{\delta, \eta},$}
\]
which follows from \eqref{eq:J2.10.5} and \eqref{eq:J2.10.10}.
Hence to prove the assertion of the lemma, it is enough to show
that $m(Q_\delta \setminus D_{\delta, \eta}) = o(  m(Q_\delta) )$
as $\delta \to 0$. We observe that 
$m(Q_\delta) = m_{d-2}(G) \cdot \half \theta \delta^2$,
where $\theta$ is the dihedral angle of $A$ at its subfacet $G$.
Let $\eps > 0$, then by \lemref{lemJ2.13} there is
 $\delta_0 = \delta_0(A,G,\eta,\eps) > 0$ such that
 for any  $\delta < \delta_0$ we have
\[
m(P_\delta \triangle Q_\delta) \, \mu(G-G) < \eps \, \eta \, m(Q_\delta).
\]
It follows that
\[
\int_{\R^d} (\1_{P_\delta \triangle Q_{\delta}} \ast \mu )\, dm 
= m(P_\delta \triangle Q_{\delta}) \int_{\R^d} d \mu <
\eps \, \eta \, m(Q_\delta),
\]
which in turn implies
\[
m(Q_\delta \setminus D_{\delta, \eta}) \leq \eta^{-1}
\int_{\R^d} (\1_{P_\delta \triangle Q_{\delta}} \ast \mu )\, dm 
< \eps  \, m(Q_\delta).
\]
This  confirms that indeed we have
$m(Q_\delta \setminus D_{\delta, \eta}) = o(  m(Q_\delta) )$
as $\delta \to 0$. 
\end{proof}


\section{Spectral convex polytopes can tile by translations, II}
\label{sect:beltii}

In this section we prove the following theorem, which is
the final result needed for the proof of
Fuglede's conjecture for convex bodies:

\begin{thm}
\label{thmJ1.0}
Let $A \subset \R^d$ be a convex polytope, which is centrally symmetric 
and has centrally symmetric facets. Assume that
the complement $A^\cm$ of $A$
 admits a weak tiling by translates of $A$,
that is, there exists a positive, locally finite measure $\mu$ such that
$\1_A \ast \mu = \1_{A^\cm}$ a.e.
Then each belt of $A$ consists of either $4$ or $6$ facets.
\end{thm}

\thmref{thmA21} is then obtained as a consequence of 
\thmref{thmA11} and \thmref{thmJ1.0}.

The rest of the section is devoted to the proof of  \thmref{thmJ1.0}.

\subsection{}
Let $G$ be any one of the subfacets of $A$, and suppose that
 the belt of $A$  generated by $G$ has $2m$ facets. 
Let $F_0, F_1, F_2, \dots, F_{2m-1}, F_{2m} = F_0$ be
an enumeration of the facets of the belt, such that
$F_{i-1}$ and $F_{i}$ are adjacent facets for each
$1 \leq i \leq 2m$. The intersection $F_{i-1} \cap F_{i}$
of any pair of consecutive facets in the belt is then a translate
 of either $G$ or $-G$. We shall suppose that $G$ itself is given by
$G = F_1 \cap F_2$.

Our goal is to  show that
under the assumptions in \thmref{thmJ1.0}, the belt 
can  have only $4$ or $6$ facets, that is, we must have
$m \leq 3$.

The belt generated by $G$ consists of $m$ pairs of
opposite facets $\{F_i, F_{i+m}\}$
$(0 \leq i \leq m-1)$. We can assume, with no loss of generality, that $A$ is
symmetric with respect to the origin, that is, $A = - A$.
Then for each facet $F_i$ in the belt, its 
opposite facet is given by $ - F_i$. Since the facets of $A$
are centrally symmetric, $-F_i$ is a translate of $F_i$, so
there is a  translation vector $\tau_i$ which carries
$-F_i$ onto $F_i$, that is, $F_i = -F_i + \tau_i$.

\subsection{}
Recall that we have assumed that 
the complement $A^\cm$ of $A$
 admits a weak tiling by translates of $A$. This means
that  there exists a positive, locally finite measure $\mu$ such that
$\1_A \ast \mu = \1_{A^\cm}$ a.e.
For each $0 \leq i \leq 2m$
we define 
\begin{equation}
\label{eq:J2.1.6}
T_i := \relint(F_i - F_i) + \tau_i, \quad
\mu'_i := \mu \cdot \1_{T_i}, \quad
\nu'_i := \mu'_i \ast \delta_{-\tau_i}.
\end{equation}

We observe that $\supp(\mu'_i)$ is contained in the hyperplane
passing through
the point $\tau_i$ and which is parallel to the facet $F_i$, while
$\supp(\nu'_i)$ is contained 
in the hyperplane  through the origin which is parallel to $F_i$.

We also notice that we have $-F_i = F_i - \tau_i$ and hence
$T_i = 2 \relint(F_i)$. This implies that the sets $T_i$ 
$(0 \leq i \leq 2m-1)$ are
pairwise disjoint, because any two distinct facets of $A$ have disjoint 
relative interiors.

\begin{lem}
\label{lemJ1.2}
For each $0 \leq i \leq 2m$ we have
 $\1_{F_i} \ast \nu'_i \geq  1$ a.e.\ with respect to
the $(d-1)$-dimensional volume measure on the facet $F_i$.
\end{lem}

\begin{proof}
Let $A' := A + \tau_i$ and $B' := A$.  Then $A'$ and $B'$ 
are two  convex polytopes in $\R^d$
 with nonempty, disjoint interiors, and they share $F_i$ as 
a common facet.  Let $\rho := \mu \ast \delta_{-\tau_i}$, then we have
$\1_{A'} \ast \rho = \1_{A^\cm}$ a.e.,
and in particular, 
$\1_{A'} \ast \rho \geq 1$ a.e.\ on $A'$, and
$\1_{A'} \ast \rho = 0$ a.e.\ on $B'$. Since $\nu'_i$ is the
restriction of the measure $\rho$ to
$\relint(F_i - F_i)$, we can apply \lemref{lemJ2.7} 
and conclude that
$\1_{F_i} \ast \nu'_i \geq 1$ a.e.\ with respect to
the $(d-1)$-dimensional volume measure on the facet $F_i$,
as we had to show.
\end{proof}

\subsection{}
Let $C_i := \convex\{F_i, F_i - \tau_i\}$ be
the prism contained in $A$ with bases $F_i$ and $F_i - \tau_i = -F_i$.
We  also define the prism 
$D_i :=  C_i + \tau_i = \convex\{F_i,F_i + \tau_i\}$.

\begin{lem}
\label{lemJ1.3}
For each $0 \leq i \leq 2m$ we have
 $\1_{C_i} \ast \mu'_i \geq 1$ a.e.\ on $D_i$.
\end{lem}

\begin{proof}
By \lemref{lemJ1.2} the facet $F_i$ has a subset
$E_i$ of full $(d-1)$-dimensional volume measure,
 such that we have
$(\1_{F_i} \ast \nu'_i)(y) \geq 1$ for every $y \in E_i$.
By  Fubini's theorem, the set of all points $x$  of the form
$x = y + \lam \tau_i$, $y \in E_i$, $0 \leq \lam \leq 1$,
constitutes a subset $D'_i$ of $D_i$
of full $d$-dimensional volume measure. 
Since $\supp(\nu'_i)$ is contained 
in the hyperplane  through the origin  parallel to $F_i$,
it follows that for any point $x$ of the above form we have
 $(\1_{D_i} \ast \nu'_i)(x) = (\1_{F_i} \ast \nu'_i)(y) \geq 1$.
We obtain that
 $\1_{D_i} \ast \nu'_i \geq 1$ a.e.\ on $D_i$. But this is
equivalent to the assertion of the lemma.
\end{proof}

In order to better understand the assertion of \lemref{lemJ1.3}, observe that
the prism $D_i$ is contained (up to measure zero)
in the complement $A^\cm$ of $A$, and hence we trivially have
the ``weak covering'' property  $\1_A \ast \mu \geq 1$ a.e.\ on $D_i$.
However the point of the lemma is that only the part of $A$
which lies in the prism $C_i$, and only the part of the measure $\mu$
which lies in  $T_i$, can in fact contribute to
 this covering (see Figure \ref{fig:img0945}).


\begin{figure}[ht]
\centering

\begin{tikzpicture}[scale=0.0375, style=mystyle]

\def\ax{100};
\def\ay{100};
\def\bx{142};
\def\by{82};
\def\cx{160};
\def\cy{47}
\def\dx{153};
\def\dy{25};
\def\ex{148};
\def\ey{15};
\def\zx{63};
\def\zy{96};
\def\ox{92.5};
\def\oy{47.5};
\def\aax{2*\ox - \ax}
\def\aay{2*\oy - \ay}
\def\bbx{2*\ox - \bx}
\def\bby{2*\oy - \by}
\def\ccx{2*\ox - \cx}
\def\ccy{2*\oy - \cy}
\def\ddx{2*\ox - \dx}
\def\ddy{2*\oy - \dy}
\def\eex{2*\ox - \ex}
\def\eey{2*\oy - \ey}
\def\zzx{2*\ox - \zx}
\def\zzy{2*\oy - \zy}
\def\pix{{2*\ax -  (\bbx)}}
\def\piy{{2*\ay -  (\bby)}}
\def\qix{{2*\bx -  (\aax)}}
\def\qiy{{2*\by -  (\aay)}}
\def\pjx{{2*\ax -  (\zzx)}}
\def\pjy{{2*\ay -  (\zzy)}}
\def\qjx{{2*\zx -  (\aax)}}
\def\qjy{{2*\zy -  (\aay)}}
\def\hiax{\ax*3 - 2*\bx}
\def\hiay{\ay*3 - 2*\by}
\def\hibx{-1.5*\ax + 2.5* \bx}
\def\hiby{-1.5*\ay + 2.5* \by}
\def\hjax{\ax*3.5 - 2.5*\zx}
\def\hjay{\ay*3.5 - 2.5*\zy}
\def\hjbx{-1.5*\ax + 2.5* \zx}
\def\hjby{-1.5*\ay + 2.5* \zy}
\def\cix{{0.5*\ax + 0.5*(\bbx)}}
\def\ciy{{0.5*\ay + 0.5*(\bby)}}
\def\dix{{1.75*\ax - 0.75*(\bbx)}}
\def\diy{{1.75*\ay - 0.75*(\bby)}}
\def\tauix{\ax+\bx-\ox}
\def\tauiy{\ay+\by-\oy}
\def\tax{2*\ax-\ox}
\def\tay{2*\ay-\oy}
\def\tbx{2*\bx-\ox}
\def\tby{2*\by-\oy}

\coordinate (O) at (\ox,\oy);
\coordinate (A) at (\ax,\ay);
\coordinate (B) at (\bx,\by);
\coordinate (C) at (\cx,\cy);
\coordinate (D) at (\dx,\dy);
\coordinate (E) at (\ex,\ey);
\coordinate (Z) at (\zx,\zy);
\coordinate (AA) at (\aax,\aay);
\coordinate (BB) at (\bbx,\bby);
\coordinate (CC) at (\ccx,\ccy);
\coordinate (DD) at (\ddx,\ddy);
\coordinate (EE) at (\eex,\eey);
\coordinate (ZZ) at (\zzx,\zzy);
\coordinate (PI) at (\pix,\piy);
\coordinate (QI) at (\qix,\qiy);
\coordinate (PJ) at (\pjx,\pjy);
\coordinate (QJ) at (\qjx,\qjy);
\coordinate (HIA) at (\hiax,\hiay);
\coordinate (HIB) at (\hibx,\hiby);
\coordinate (HJA) at (\hjax,\hjay);
\coordinate (HJB) at (\hjbx,\hjby);
\coordinate (CI) at (\cix,\ciy);
\coordinate (DI) at (\dix,\diy);
\coordinate (TAUI) at (\tauix,\tauiy);
\coordinate (TA) at (\tax,\tay);
\coordinate (TB) at (\tbx,\tby);

\draw [myyellow, mydashedb] (O) -- (TA);
\draw [myyellow, mydashedb] (O) -- (TB);

\draw [mygreen] (BB) -- (A) -- (PI) -- (QI) -- (B) -- (AA) ;

\draw(Z) -- (A) -- (B) -- (C);
\draw(ZZ) -- (AA) -- (BB) -- (CC);
\draw [mydasheda]   (C) .. controls (D) and (E) .. (ZZ);
\draw [mydasheda]   (CC) .. controls (DD) and (EE) .. (Z);

\fill (O) circle (\mycirc);
\fill (A) circle (\mycirc);
\fill (B) circle (\mycirc);
\fill (C) circle (\mycirc);
\fill (Z) circle (\mycirc);
\fill (AA) circle (\mycirc);
\fill (BB) circle (\mycirc);
\fill (CC) circle (\mycirc);
\fill (ZZ) circle (\mycirc);

\draw [myred] (TA) -- (TB);
\draw [myred,fill=white] (TA) circle  (\mycirc);
\draw [myred,fill=white] (TB) circle  (\mycirc);
\draw [myred,fill=myred] (TAUI) circle (\mycirc);

\fill [mygreen] (PI) circle (\mycirc);
\fill [mygreen] (QI) circle (\mycirc);

\draw [myred] (TAUI) 
	node[xshift=-0.2cm,yshift=-0.3cm] 
	{$\boldsymbol{\tau_i}$};

\draw [mygreen] (CI) 
	node[xshift=-0.35cm,yshift=0.25cm] 
	{$\boldsymbol{C_i}$};

\draw [mygreen] (DI) 
	node[xshift=-0.35cm,yshift=0.25cm] 
	{$\boldsymbol{D_i}$};

\draw [decorate,decoration={brace,raise=6pt,amplitude=7pt,
	pre=moveto, pre length=0.15cm,
	post=moveto, post length=0.15cm}, mypurple]
	(TA) -- (TB) 
	node [mypurple,midway,xshift=0.45cm,yshift=0.65cm,rotate=-21] 
	{$\boldsymbol{T_i}$};

\draw [decorate,decoration={brace,raise=5pt,amplitude=4pt,
	pre=moveto, pre length=0.25cm,
	post=moveto, post length=0.15cm}, myred]
	(B) -- (A) 
	node [myred,midway,xshift=-0.325cm,yshift=-0.565cm,rotate=-21] 
	{$\boldsymbol{F_i}$};

\draw [decorate,decoration={brace,raise=5pt,amplitude=4pt,
	pre=moveto, pre length=0.15cm,
	post=moveto, post length=0.15cm}, myred]
	(AA) -- (BB) 
	node [myred,midway,xshift=-0.295cm,yshift=-0.6cm,rotate=-21] 
	{$\boldsymbol{F_i - \tau_i}$};

\end{tikzpicture}

\caption{According to \lemref{lemJ1.3}, the prism $D_i$ is 
	``weakly covered'' by the translates of the prism $C_i$ 
	with respect to the measure $\mu'_i$.}
\label{fig:img0945}
\end{figure}
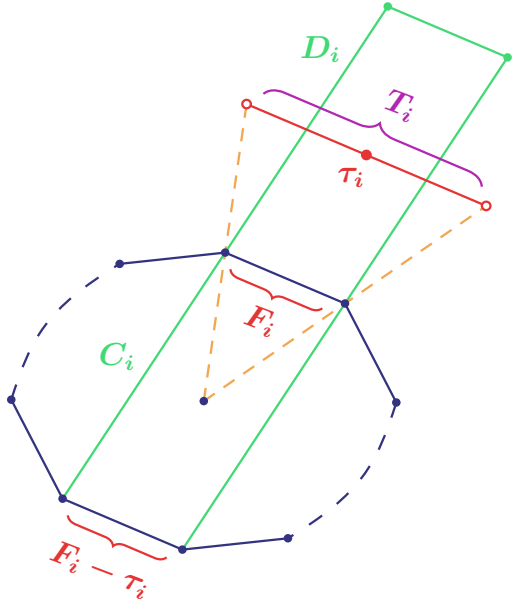


\subsection{}
Let us recall that 
$F_0, F_1, F_2, \dots, F_{2m-1}, F_{2m} = F_0$ 
is  an enumeration of the facets of the belt generated 
by the subfacet $G$ of the convex polytope $A$,
and we have  assumed that the subfacet
 $G$ is given by $G = F_1 \cap F_2$.

For  $i \in \{1,2\}$ (and only for these two values of $i$)
 we  denote
\begin{equation}
\label{eq:J2.1.10}
S_i := \relint(G - G) + \tau_i, \quad \mu''_i := \mu'_i \cdot \1_{S_i}, \quad
\nu''_i := \mu''_i \ast \delta_{- \tau_i}.
\end{equation}

We notice that  $S_i$ is a subset of the set $T_i$ defined
in \eqref{eq:J2.1.6}. This  follows from \lemref{lemJ2.14}
applied relative to the affine hull of the facet $F_i$.
In particular, this shows that
$S_1$ and $S_2$ are disjoint sets, because the sets
$T_1$ and $T_2$ are disjoint. 

It also follows that we have
$\mu''_i = \mu \cdot \1_{S_i}$, that is, in the definition
of the measure $\mu''_i$ it does not
matter whether we restrict  $\mu'_i$ or $\mu$ to the
set $S_i$.

\begin{lem}
\label{lemJ1.4}
Assume that the belt of $A$ generated  by the subfacet $G$ has $6$ or more facets.
Then  for each  $i \in \{1,2\}$ we have
$\1_{G} \ast \nu''_i \geq  1$ a.e.\ with respect to
the $(d-2)$-dimensional volume measure on the subfacet $G$.
\end{lem}

\begin{proof}
We suppose that $i$ is an element of the set $\{1,2\}$ and we
let $j$ be the other element, so that $(i,j) = (1,2)$ or $(2,1)$.
 The proof is divided into several steps.

\emph{Step 1}:
We first claim that if $t \in \supp(\nu'_i)$, then
$F_i + t$ cannot intersect the interior 
of the prism $D_j$. For suppose that $(F_i + t) \cap \interior(D_j)$ is nonempty.
Let $s := t + \tau_i$, then $s \in \supp(\mu'_i)$ and we have
$(F_i - \tau_i + s) \cap \interior(D_j)$ is nonempty.
Since $F_i - \tau_i + s$ is a facet of the prism $C_i + s$, it follows that
$(C_i+s)\cap D_j$ has nonempty interior, and in particular we have
 $m((C_i+s)\cap D_j) > 0$. By \lemref{lemA8.2} this implies that
$\1_{C_i} \ast \mu'_i$ cannot vanish a.e.\ on $D_j$,
and hence there exist $\eta > 0$ and a set $E \subset D_j$, $m(E)>0$,
such that $\1_{C_i} \ast \mu'_i \geq \eta$ a.e.\ on $E$.
 On the other hand we have $\1_{C_j} \ast \mu'_j \geq  \1_{D_j}$ 
a.e.\ by \lemref{lemJ1.3}. Using the fact that $C_i$ and $C_j$
are both subsets of $A$, and that $\mu'_i$ and $\mu'_j$ are the
restrictions of $\mu$ to the two disjoint sets $T_i$ and $T_j$
respectively, we conclude that
\[
\1_A \ast \mu \geq \1_{C_i} \ast \mu'_i + 
 \1_{C_j} \ast \mu'_j \geq  \eta \, \1_E + \1_{D_j} \geq (1 + \eta)\, \1_E
\quad \text{a.e.}
\]
However this contradicts the weak tiling assumption $\1_A \ast \mu = \1_{A^\cm}$ a.e.
This establishes our claim  that $F_i + t$ cannot intersect the interior 
of the prism $D_j$.

\emph{Step 2}:
Let $H_i$ be the hyperplane containing the facet $F_i$, and define
\[
B_i := H_i \cap D_j.
\]
We claim that $B_i$ is a $(d-1)$-dimensional convex polytope, that
$\relint(B_i) \subset \interior(D_j)$, and that
$G$ is a $(d-2)$-dimensional face of $B_i$.

First, it is clear that $B_i$ is a convex polytope, being the intersection 
of a convex polytope and a hyperplane.

Next, recall that we have assumed the belt of $A$ generated  by the subfacet $G$
 to have $6$ or more facets. This implies that the facet 
$L_j := \convex\{G, G- \tau_j\}$ of the
prism $C_j$ divides the dihedral angle of $A$ at its subfacet
$G$ into two (strictly positive) angles $\theta$ and $\varphi$, where
$\theta$ is the dihedral angle between $L_j$ and $F_i$, and
$\varphi$ is the dihedral angle between $L_j$ and $F_j$.
Hence the hyperplane $H_i$ divides the
dihedral angle of the prism $D_j$ at its subfacet
$G$ into two strictly positive angles $\theta$ and $\pi - \theta - \varphi$.
It follows that $H_i$ must intersect the interior of the prism $D_j$ and so
$B_i$ is a $(d-1)$-dimensional convex polytope
(see Figure \ref{fig:img0944}).


\begin{figure}[ht]
\centering

\begin{tikzpicture}[scale=0.0375, style=mystyle]

\def\ax{100};
\def\ay{100};
\def\bx{142};
\def\by{82};
\def\cx{160};
\def\cy{47}
\def\dx{153};
\def\dy{25};
\def\ex{148};
\def\ey{15};
\def\zx{63};
\def\zy{96};
\def\ox{92.5};
\def\oy{47.5};
\def\aax{2*\ox - \ax}
\def\aay{2*\oy - \ay}
\def\bbx{2*\ox - \bx}
\def\bby{2*\oy - \by}
\def\ccx{2*\ox - \cx}
\def\ccy{2*\oy - \cy}
\def\ddx{2*\ox - \dx}
\def\ddy{2*\oy - \dy}
\def\eex{2*\ox - \ex}
\def\eey{2*\oy - \ey}
\def\zzx{2*\ox - \zx}
\def\zzy{2*\oy - \zy}
\def\pjx{{2*\ax -  (\zzx)}}
\def\pjy{{2*\ay -  (\zzy)}}
\def\qjx{{2*\zx -  (\aax)}}
\def\qjy{{2*\zy -  (\aay)}}
\def\hiax{\ax*3 - 2*\bx}
\def\hiay{\ay*3 - 2*\by}
\def\hibx{-1.5*\ax + 2.5* \bx}
\def\hiby{-1.5*\ay + 2.5* \by}
\def\hjax{\ax*3.5 - 2.5*\zx}
\def\hjay{\ay*3.5 - 2.5*\zy}
\def\hjbx{-1.5*\ax + 2.5* \zx}
\def\hjby{-1.5*\ay + 2.5* \zy}
\def\cjx{{0.5*\zx + 0.5*(\aax)}}
\def\cjy{{0.5*\zy + 0.5*(\aay)}}
\def\djx{{1.5*\ax - 0.5*(\zzx)}}
\def\djy{{1.5*\ay - 0.5*(\zzy)}}

\coordinate (O) at (\ox,\oy);
\coordinate (A) at (\ax,\ay);
\coordinate (B) at (\bx,\by);
\coordinate (C) at (\cx,\cy);
\coordinate (D) at (\dx,\dy);
\coordinate (E) at (\ex,\ey);
\coordinate (Z) at (\zx,\zy);
\coordinate (AA) at (\aax,\aay);
\coordinate (BB) at (\bbx,\bby);
\coordinate (CC) at (\ccx,\ccy);
\coordinate (DD) at (\ddx,\ddy);
\coordinate (EE) at (\eex,\eey);
\coordinate (ZZ) at (\zzx,\zzy);
\coordinate (PJ) at (\pjx,\pjy);
\coordinate (QJ) at (\qjx,\qjy);
\coordinate (HIA) at (\hiax,\hiay);
\coordinate (HIB) at (\hibx,\hiby);
\coordinate (HJA) at (\hjax,\hjay);
\coordinate (HJB) at (\hjbx,\hjby);
\coordinate (CJ) at (\cjx,\cjy);
\coordinate (DJ) at (\djx,\djy);

\draw [name path=linehi, myyellow, mydashedb] (HIA) -- (HIB);
\draw [name path=linehj, myyellow, mydashedb] (HJA) -- (HJB);

\draw [mygreen] (AA) -- (Z) -- (QJ) -- (PJ) -- (A) -- (ZZ) ;

\draw(Z) -- (A) -- (B) -- (C);
\draw(ZZ) -- (AA) -- (BB) -- (CC);
\draw [mydasheda]   (C) .. controls (D) and (E) .. (ZZ);
\draw [mydasheda]   (CC) .. controls (DD) and (EE) .. (Z);

\fill (O) circle (\mycirc);
\fill (A) circle (\mycirc);
\fill (B) circle (\mycirc);
\fill (C) circle (\mycirc);
\fill (Z) circle (\mycirc);
\fill (AA) circle (\mycirc);
\fill (BB) circle (\mycirc);
\fill (CC) circle (\mycirc);
\fill (ZZ) circle (\mycirc);

\fill [mygreen] (PJ) circle (\mycirc);
\fill [mygreen] (QJ) circle (\mycirc);

\draw (A) 
	node[xshift=0.275cm,yshift=0.375cm]  
	{$\boldsymbol{G}$};

\draw [myyellow]	 (HJA) 
	node[xshift=-0.25cm,yshift=0.3cm,rotate=6] 
	{$\boldsymbol{H_j}$};

\draw [myyellow] (HIB) 
	node[xshift=-0.25cm,yshift=0.55cm,rotate=-21] 
	{$\boldsymbol{H_i}$};

\draw [mygreen] (CJ) 
	node[xshift=-0.5cm,yshift=-0.35cm] 
	{$\boldsymbol{C_j}$};

\draw [mygreen] (DJ) 
	node[xshift=0.5cm,yshift=0.35cm] 
	{$\boldsymbol{D_j}$};

\draw [decorate,decoration={brace,raise=5pt,amplitude=4pt,
	pre=moveto, pre length=0.15cm,
	post=moveto, post length=0.15cm}, myred]
	(A) -- (Z) 
	node [myred,midway,xshift=0.1cm,yshift=-0.7cm,rotate=6] 
	{$\boldsymbol{F_j}$};

\draw [decorate,decoration={brace,raise=5pt,amplitude=4pt,aspect=0.4,
	pre=moveto, pre length=0.15cm,
	post=moveto, post length=0.35cm}, myred]
	(B) -- (A) 
	node [myred,midway,xshift=-0.1cm,yshift=-0.7cm,rotate=-21] 
	{$\boldsymbol{F_i}$};

\draw [decorate,decoration={brace,raise=5pt,amplitude=4pt,aspect=0.4,
	pre=moveto, pre length=0cm,
	post=moveto, post length=0.3cm}, myred]
	(intersection of HIA--HIB and Z--QJ) -- (A) 
	node [myred,midway,xshift=0cm,yshift=0.7cm,rotate=-21] 
	{$\boldsymbol{B_i}$};

\end{tikzpicture}

\caption{If the belt generated by the subfacet $G$ has $6$ or more
	facets, then the hyperplane $H_i$ containing the
	facet $F_i$ intersects the interior of the prism $D_j$,
	and $B_i = H_i \cap D_j$ is a $(d-1)$-dimensional convex polytope.}
\label{fig:img0944}
\end{figure}
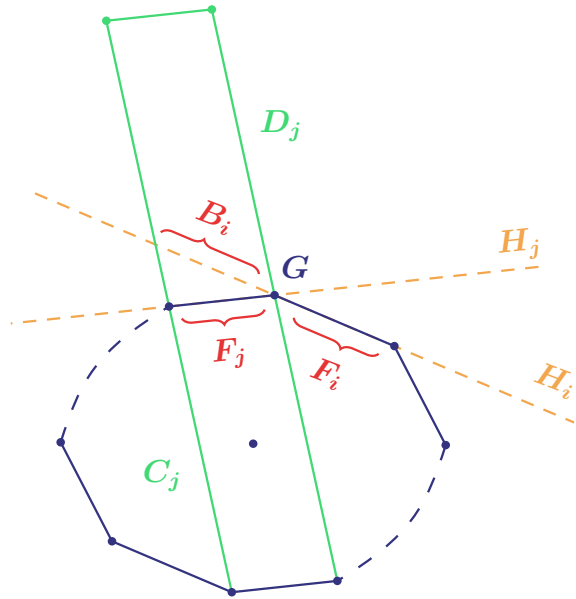


Let $H_j$ be the hyperplane containing the facet $F_j$, and $H^{+}_j$
be the closed halfspace bounded by $H_j$ which contains the prism $D_j$.
 Then $B_i$ is contained in  $H^{+}_j$ and we have $B_i \cap H_j = G$.
This shows that $H_j$ is a support hyperplane of $B_i$ and 
$G$ is a face (of dimension $d-2$) of $B_i$.

Finally, we show that $\relint(B_i) \subset \interior(D_j)$. Indeed,
let $x \in \relint(B_i)$. Since $H_i$ is the affine hull of $B_i$, 
this means that
there is an open set $V$ such that $x \in V \cap H_i 
\subset B_i$. In particular this implies that $x \in D_j$,
so it is enough to prove that $x$ cannot lie on $\bd{D_j}$. 
Indeed, since $x \in V \cap H_i \subset D_j$, the point $x$ can lie
on $\bd{D_j}$ only if $H_i$ is a support hyperplane of $D_j$.
But this is not the case, since $H_i$ intersects $\interior(D_j)$,
so we must have $x \in \interior(D_j)$.

\emph{Step 3}:
It follows that if $t \in \supp(\nu'_i)$, then $F_i+t$ cannot 
intersect $\relint(B_i)$. Indeed, as we have shown above, $\relint(B_i)$
is contained in the interior 
of the prism $D_j$, while  $F_i+t$ cannot intersect the interior 
of $D_j$. Hence $F_i+t$ and $\relint(B_i)$
are disjoint sets for every  $t \in \supp(\nu'_i)$.
From this we conclude that
$\1_{F_i} \ast \nu'_i = 0$ a.e.\ with respect to
the $(d-1)$-dimensional volume measure on $B_i$.

\emph{Step 4}:
The sets $F_i$ and $B_i$ are two
$(d-1)$-dimensional convex polytopes contained in the
same hyperplane $H_i$, they have disjoint relative
interiors,  and they share
$G$ as a common $(d-2)$-dimensional face.
Recall that by \lemref{lemJ1.2} we have
 $\1_{F_i} \ast \nu'_i \geq  1$ a.e.\ with respect to
the $(d-1)$-dimensional volume measure on $F_i$, while
we have just shown in Step 3 above that
$\1_{F_i} \ast \nu'_i = 0$ a.e.\ with respect to
the $(d-1)$-dimensional volume measure on $B_i$.
We may therefore apply \lemref{lemJ2.7} (relative to the
hyperplane $H_i$ containing $F_i$ and $B_i$) and conclude
that the measure $\nu''_i$ obtained by
restricting  the measure $\nu'_i$ to
$\relint(G - G)$, satisfies 
$\1_{G} \ast \nu''_i \geq 1$ a.e.\ with respect to
the $(d-2)$-dimensional volume measure on $G$.
So we obtain the assertion of the lemma.
\end{proof}

\subsection{}
For $i \in \{1,2\}$, we denote
by $N_i$ the  hyperplane which contains the
subfacet $G$ and which is parallel to the facet
$F_0$ if $i=1$, or parallel to the facet $F_3$ if $i=2$.
Let $N^{-}_i$ be the closed halfspace bounded by  $N_i$ 
that has exterior normal unit vector which is opposite to
the exterior normal vector of $A$ at the facet $F_0$ (if $i=1$) or $F_3$ (if $i=2$).
It is not difficult to verify that  $N^{-}_i$ is the support
halfspace of $A + \tau_i$ at its facet
$F_0 - \tau_0 + \tau_1$ for $i=1$, or 
$F_3 - \tau_3 + \tau_2$ for $i=2$.

\begin{lem}
\label{lemJ1.5}
Assume that the belt of $A$ generated  by the subfacet $G$ has $8$ or more facets.
Then the set $M := (A+ \tau_1) \cap (A+ \tau_2)$  is a convex polytope with nonempty
interior, $G$ is a subfacet of $M$, and $N^{-}_1$, $N^{-}_2$ 
are the support halfspaces of $M$ at its two facets which
meet at the subfacet $G$ (see Figure \ref{fig:img0926}).
\end{lem}


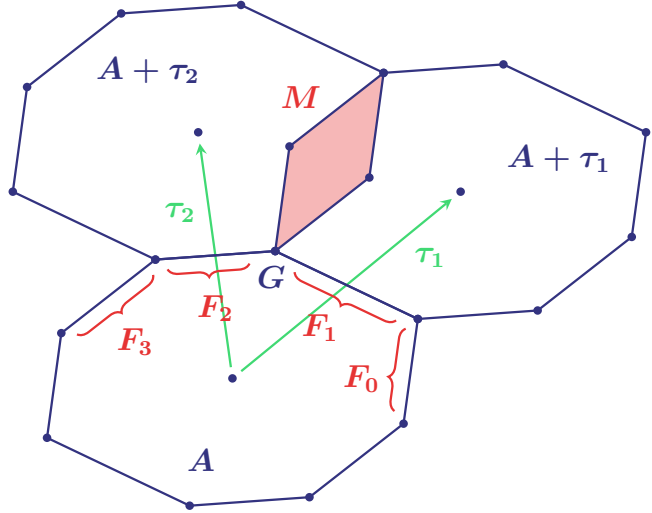
\begin{figure}[ht]
\centering

\begin{tikzpicture}[scale=0.0375, style=mystyle]

\def\ax{100};
\def\ay{100};
\def\bx{150};
\def\by{76};
\def\cx{145};
\def\cy{39}
\def\zx{58};
\def\zy{97};
\def\ox{85};
\def\oy{55};
\def\aax{2*\ox - \ax}
\def\aay{2*\oy - \ay}
\def\bbx{2*\ox - \bx}
\def\bby{2*\oy - \by}
\def\ccx{2*\ox - \cx}
\def\ccy{2*\oy - \cy}
\def\zzx{2*\ox - \zx}
\def\zzy{2*\oy - \zy}
\def\tauix{\ax+\bx-2*\ox}
\def\tauiy{\ay+\by-2*\oy}
\def\taujx{\ax+\zx-2*\ox}
\def\taujy{\ay+\zy-2*\oy}

\coordinate (O) at (\ox,\oy);
\coordinate (A) at (\ax,\ay);
\coordinate (B) at (\bx,\by);
\coordinate (C) at (\cx,\cy);
\coordinate (Z) at (\zx,\zy);
\coordinate (AA) at (\aax,\aay);
\coordinate (BB) at (\bbx,\bby);
\coordinate (CC) at (\ccx,\ccy);
\coordinate (ZZ) at (\zzx,\zzy);

\coordinate (TAUI) at (\tauix,\tauiy);
\coordinate (TAUJ) at (\taujx,\taujy);

\coordinate (OI) at (\ox + \tauix ,\oy + \tauiy);
\coordinate (AI) at (\ax + \tauix ,\ay + \tauiy);
\coordinate (BI) at (\bx + \tauix,\by + \tauiy);
\coordinate (CI) at (\cx + \tauix,\cy + \tauiy);
\coordinate (ZI) at (\zx + \tauix,\zy + \tauiy);
\coordinate (AAI) at (\aax + \tauix,\aay + \tauiy);
\coordinate (BBI) at (\bbx + \tauix,\bby + \tauiy);
\coordinate (CCI) at (\ccx + \tauix,\ccy + \tauiy);
\coordinate (ZZI) at (\zzx + \tauix,\zzy + \tauiy);

\coordinate (OJ) at (\ox + \taujx ,\oy + \taujy);
\coordinate (AJ) at (\ax + \taujx ,\ay + \taujy);
\coordinate (BJ) at (\bx + \taujx,\by + \taujy);
\coordinate (CJ) at (\cx + \taujx,\cy + \taujy);
\coordinate (ZJ) at (\zx + \taujx,\zy + \taujy);
\coordinate (AAJ) at (\aax + \taujx,\aay + \taujy);
\coordinate (BBJ) at (\bbx + \taujx,\bby + \taujy);
\coordinate (CCJ) at (\ccx + \taujx,\ccy + \taujy);
\coordinate (ZZJ) at (\zzx + \taujx,\zzy + \taujy);

\fill [myred, opacity=0.35]
	(A) -- (CCI) -- (ZI) -- (CJ) -- (A);

\draw [->,>=stealth,shorten >=0.15cm,shorten <=0.15cm,mygreen] 
	(O) -- (OI);

\draw [->,>=stealth,shorten >=0.15cm,shorten <=0.15cm,mygreen] 
	(O) -- (OJ);

\draw (CC) -- (Z) -- (A) -- (B) -- (C) -- (ZZ) -- (AA) -- (BB) -- (CC);
\draw (CCI) -- (ZI) -- (AI) -- (BI) -- (CI) -- (ZZI) -- (AAI) -- (BBI) -- (CCI);
\draw (CCJ) -- (ZJ) -- (AJ) -- (BJ) -- (CJ) -- (ZZJ) -- (AAJ) -- (BBJ) -- (CCJ);

\fill (O) circle (\mycirc);
\fill (A) circle (\mycirc);
\fill (B) circle (\mycirc);
\fill (C) circle (\mycirc);
\fill (Z) circle (\mycirc);
\fill (AA) circle (\mycirc);
\fill (BB) circle (\mycirc);
\fill (CC) circle (\mycirc);
\fill (ZZ) circle (\mycirc);

\fill (OI) circle (\mycirc);
\fill (AI) circle (\mycirc);
\fill (BI) circle (\mycirc);
\fill (CI) circle (\mycirc);
\fill (ZI) circle (\mycirc);
\fill (AAI) circle (\mycirc);
\fill (BBI) circle (\mycirc);
\fill (CCI) circle (\mycirc);
\fill (ZZI) circle (\mycirc);

\fill (OJ) circle (\mycirc);
\fill (AJ) circle (\mycirc);
\fill (BJ) circle (\mycirc);
\fill (CJ) circle (\mycirc);
\fill (ZJ) circle (\mycirc);
\fill (AAJ) circle (\mycirc);
\fill (BBJ) circle (\mycirc);
\fill (CCJ) circle (\mycirc);
\fill (ZZJ) circle (\mycirc);

\draw (A) 
	node[xshift=-0.05cm,yshift=-0.35cm] 
	{$\boldsymbol{G}$};

\draw (AA) 
	node[xshift=0.15cm,yshift=0.6cm] 
	{$\boldsymbol{A}$};

\draw (BI) 
	node[xshift=-1.15cm,yshift=-0.35cm] 
	{$\boldsymbol{A+\tau_1}$};

\draw (ZJ) 
	node[xshift=0.35cm,yshift=-0.75cm] 
	{$\boldsymbol{A+\tau_2}$};

\draw [mygreen] ({\ox+0.75*(\tauix)}, {\oy+0.75*(\tauiy)}) 
	node[xshift=0.35cm,yshift=-0.25cm] 
	{$\boldsymbol{\tau_1}$};

\draw [mygreen] ({\ox+0.75*(\taujx)}, {\oy+0.75*(\taujy)}) 
	node[xshift=-0.35cm,yshift=-0.25cm] 
	{$\boldsymbol{\tau_2}$};

\draw [myred] (CCI) 
	node[xshift=0.15cm,yshift=0.65cm] 
	{$\boldsymbol{M}$};

\draw [decorate,decoration={brace,raise=4pt,amplitude=4pt,aspect=0.4,
	pre=moveto, pre length=0.15cm,
	post=moveto, post length=0.15cm}, myred]
	(C) -- (B) 
	node [myred,midway,xshift=-0.65cm,yshift=-0.1cm] 
	{$\boldsymbol{F_0}$};

\draw [decorate,decoration={brace,raise=4pt,amplitude=4pt,
	pre=moveto, pre length=0.25cm,
	post=moveto, post length=0.35cm}, myred]
	(B) -- (A) 
	node [myred,midway,xshift=-0.35cm,yshift=-0.565cm] 
	{$\boldsymbol{F_1}$};

\draw [decorate,decoration={brace,raise=4pt,amplitude=4pt,
	pre=moveto, pre length=0.35cm,
	post=moveto, post length=0.15cm}, myred]
	(A) -- (Z) 
	node [myred,midway,xshift=0cm,yshift=-0.7cm] 
	{$\boldsymbol{F_2}$};

\draw [decorate,decoration={brace,raise=4pt,amplitude=4pt,
	pre=moveto, pre length=0.15cm,
	post=moveto, post length=0.15cm}, myred]
	(Z) -- (CC) 
	node [myred,midway,xshift=0.35cm,yshift=-0.6cm] 
	{$\boldsymbol{F_3}$};

\end{tikzpicture}

\caption{The shaded region in the illustration represents the 
	convex polytope $M = (A+ \tau_1) \cap (A+ \tau_2)$ in
	\lemref{lemJ1.5}.}
\label{fig:img0926}
\end{figure}


\begin{proof}
Let $\alpha$, $\beta$ and $\gamma$ denote the dihedral angles of $A$
 at the subfacets $G$, $-G + \tau_1$ and 	$-G + \tau_2$ respectively.
If the belt of $A$ generated  by   $G$ has $8$ or more facets,
then we must have $\alpha + \beta + \gamma > 2 \pi$
(see Figure \ref{fig:img0946}). On the other hand, 
 each one of $\alpha$, $\beta$ and $\gamma$ is strictly smaller 
than $\pi$.
Hence the hyperplane $N_1$ divides the dihedral angle
of $A+\tau_2$ at the subfacet $G$
 into two strictly positive angles
$2\pi - \alpha - \beta$ and $\alpha + \beta + \gamma - 2 \pi$,
 while the hyperplane $N_2$ 
divides the dihedral angle of $A+\tau_1$ at $G$
 into two strictly positive angles
$2\pi - \alpha - \gamma$ and $\alpha + \beta + \gamma - 2 \pi$.
In particular the two hyperplanes $N_1$ and $N_2$
 are not parallel, and thus $N_1 \cap N_2 = \aff(G)$.

It is clear that  $M := (A+ \tau_1) \cap (A+ \tau_2)$  is a convex polytope,
being the intersection of two convex polytopes. 
Since $F_1$ is a facet of $A+\tau_1$ and $F_2$ is a facets of $A+\tau_2$,
then $G$  is a subfacet of both $A+\tau_1$ and $A+\tau_2$.
Let $E$ be a 
closed $(d-2)$-dimensional ball contained in $\relint(G)$. For $\delta > 0$
we denote $D(E, \delta) := (E + S_\delta) \cap N^{-}_1 \cap N^{-}_2$,
where $S_\delta$ is a closed $2$-dimensional ball of radius $\delta$  
centered at the origin and orthogonal to $\aff(G)$. Using \lemref{lemJ2.6}
we obtain that if $\delta = \delta(A,G,E) > 0$ is small enough,  
then $D(E, \delta)$ is contained  both in $A+ \tau_1$ and in $A+ \tau_2$, 
and hence $D(E, \delta) \subset M$.
It follows that $M$ has nonempty interior, and that $N_1$, $N_2$ 
are support hyperplanes of $M$ such that the corresponding support
sets $M \cap N_1$ and $M \cap N_2$ are $(d-1)$-dimensional, and hence
these support sets are facets of $M$. We conclude that
$G$ is a subfacet of $M$, being the intersection of two adjacent
facets $M \cap N_1$ and $M \cap N_2$ of $M$, and that
$N^{-}_1$, $N^{-}_2$ 
are the support halfspaces of $M$ at its two facets which
meet at the subfacet $G$.
\end{proof}


\begin{figure}[ht]
\centering

\begin{tikzpicture}[scale=0.0375, style=mystyle]

\def\ax{100};
\def\ay{100};
\def\bx{150};
\def\by{76};
\def\cx{145};
\def\cy{39}
\def\zx{58};
\def\zy{97};
\def\ox{85};
\def\oy{55};
\def\aax{2*\ox - \ax}
\def\aay{2*\oy - \ay}
\def\bbx{2*\ox - \bx}
\def\bby{2*\oy - \by}
\def\ccx{2*\ox - \cx}
\def\ccy{2*\oy - \cy}
\def\zzx{2*\ox - \zx}
\def\zzy{2*\oy - \zy}
\def\tauix{\ax+\bx-2*\ox}
\def\tauiy{\ay+\by-2*\oy}
\def\taujx{\ax+\zx-2*\ox}
\def\taujy{\ay+\zy-2*\oy}

\coordinate (O) at (\ox,\oy);
\coordinate (A) at (\ax,\ay);
\coordinate (B) at (\bx,\by);
\coordinate (C) at (\cx,\cy);
\coordinate (Z) at (\zx,\zy);
\coordinate (AA) at (\aax,\aay);
\coordinate (BB) at (\bbx,\bby);
\coordinate (CC) at (\ccx,\ccy);
\coordinate (ZZ) at (\zzx,\zzy);

\coordinate (TAUI) at (\tauix,\tauiy);
\coordinate (TAUJ) at (\taujx,\taujy);

\coordinate (OI) at (\ox + \tauix ,\oy + \tauiy);
\coordinate (AI) at (\ax + \tauix ,\ay + \tauiy);
\coordinate (BI) at (\bx + \tauix,\by + \tauiy);
\coordinate (CI) at (\cx + \tauix,\cy + \tauiy);
\coordinate (ZI) at (\zx + \tauix,\zy + \tauiy);
\coordinate (AAI) at (\aax + \tauix,\aay + \tauiy);
\coordinate (BBI) at (\bbx + \tauix,\bby + \tauiy);
\coordinate (CCI) at (\ccx + \tauix,\ccy + \tauiy);
\coordinate (ZZI) at (\zzx + \tauix,\zzy + \tauiy);

\coordinate (OJ) at (\ox + \taujx ,\oy + \taujy);
\coordinate (AJ) at (\ax + \taujx ,\ay + \taujy);
\coordinate (BJ) at (\bx + \taujx,\by + \taujy);
\coordinate (CJ) at (\cx + \taujx,\cy + \taujy);
\coordinate (ZJ) at (\zx + \taujx,\zy + \taujy);
\coordinate (AAJ) at (\aax + \taujx,\aay + \taujy);
\coordinate (BBJ) at (\bbx + \taujx,\bby + \taujy);
\coordinate (CCJ) at (\ccx + \taujx,\ccy + \taujy);
\coordinate (ZZJ) at (\zzx + \taujx,\zzy + \taujy);

\draw [myred] 
	({\ax * 0.675 + (\zx) * 0.325}, {\ay * 0.675 + (\zy) * 0.325} )
	arc (190:330:12)
	node [midway,xshift=0cm,yshift=-0.3cm]
	{$\boldsymbol{\alpha}$};

\draw [mygreen] 
	(\ax * 0.75 + \bx * 0.25, \ay * 0.75 + \by * 0.25 )
	arc (-40:110:10)
	node [midway,xshift=0.35cm,yshift=-0.15cm]
	{$\boldsymbol{\beta}$};

\draw [mygreen] 
	(\ax * 0.25 + \bx * 0.75, \ay * 0.25 + \by * 0.75 )
	arc (-180-40:-180+110:10)
	node [midway,xshift=-0.35cm,yshift=-0.15cm]
	{$\boldsymbol{\beta}$};

\draw [myyellow] 
	({\ax * 0.75 + (\zx) * 0.25}, {\ay * 0.75 + (\zy) * 0.25} ) 
	arc (180:55:16) 
	node [midway,xshift=-0.2cm,yshift=0.15cm]
	{$\boldsymbol{\gamma}$};

\draw [myyellow] 
	({\zx * 0.75 + (\ax) * 0.25}, {\zy * 0.75 + (\ay) * 0.25} ) 
	arc (0:-180+55:16) 
	node [midway,xshift=0cm,yshift=-0.3cm]
	{$\boldsymbol{\gamma}$};

\draw (CC) -- (Z) -- (A) -- (B) -- (C) -- (ZZ) -- (AA) -- (BB) -- (CC);
\draw (CCI) -- (ZI) -- (AI) -- (BI) -- (CI) -- (ZZI) -- (AAI) -- (BBI) -- (CCI);
\draw (CCJ) -- (ZJ) -- (AJ) -- (BJ) -- (CJ) -- (ZZJ) -- (AAJ) -- (BBJ) -- (CCJ);

\fill (O) circle (\mycirc);
\fill (A) circle (\mycirc);
\fill (B) circle (\mycirc);
\fill (C) circle (\mycirc);
\fill (Z) circle (\mycirc);
\fill (AA) circle (\mycirc);
\fill (BB) circle (\mycirc);
\fill (CC) circle (\mycirc);
\fill (ZZ) circle (\mycirc);

\fill (OI) circle (\mycirc);
\fill (AI) circle (\mycirc);
\fill (BI) circle (\mycirc);
\fill (CI) circle (\mycirc);
\fill (ZI) circle (\mycirc);
\fill (AAI) circle (\mycirc);
\fill (BBI) circle (\mycirc);
\fill (CCI) circle (\mycirc);
\fill (ZZI) circle (\mycirc);

\fill (OJ) circle (\mycirc);
\fill (AJ) circle (\mycirc);
\fill (BJ) circle (\mycirc);
\fill (CJ) circle (\mycirc);
\fill (ZJ) circle (\mycirc);
\fill (AAJ) circle (\mycirc);
\fill (BBJ) circle (\mycirc);
\fill (CCJ) circle (\mycirc);
\fill (ZZJ) circle (\mycirc);

\draw (AA) 
	node[xshift=0.15cm,yshift=0.6cm] 
	{$\boldsymbol{A}$};

\draw (BI) 
	node[xshift=-1.15cm,yshift=-0.35cm] 
	{$\boldsymbol{A+\tau_1}$};

\draw (ZJ) 
	node[xshift=0.35cm,yshift=-0.75cm] 
	{$\boldsymbol{A+\tau_2}$};

\end{tikzpicture}

\caption{If the belt of $A$ generated  by the subfacet $G$ has $8$ or more facets,
	then the sum of the dihedral angles $\alpha$, $\beta$ and $\gamma$
	 at the subfacets $G$, $-G + \tau_1$ and 	$-G + \tau_2$ respectively
	is strictly greater than $2 \pi$.}
\label{fig:img0946}
\end{figure}
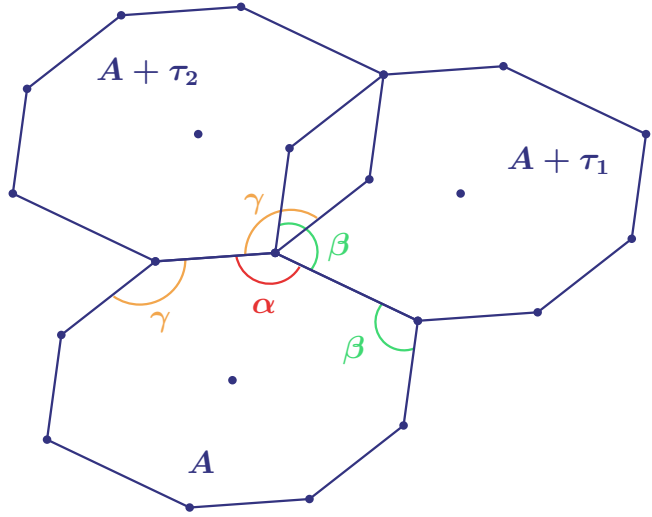


\subsection{}
Now we can finish the proof of  \thmref{thmJ1.0}.

Indeed, suppose to the contrary
 that the belt of $A$ generated  by the subfacet $G$ has $8$
 or more facets. By \lemref{lemJ1.5}, the set $M := (A+ \tau_1) \cap (A+ \tau_2)$ 
 is a convex polytope with nonempty
interior, $G$ is a subfacet of $M$, and $N^{-}_1$, $N^{-}_2$ 
are the support halfspaces of $M$ at its two facets which
meet at the subfacet $G$.
For each  $i \in \{1,2\}$,  $\nu''_i$ is a
finite measure supported on $G-G$, and by \lemref{lemJ1.4} we have
$\1_{G} \ast \nu''_i \geq  1$ a.e.\ with respect to
the $(d-2)$-dimensional volume measure on the subfacet $G$.
Then we can apply \lemref{lemJ2.10} to the convex polytope $M$.
It follows from the lemma that if we denote
$Q_\delta = (G + S_\delta) \cap N^{-}_1 \cap N^{-}_2$,
where $S_\delta$ is a closed $2$-dimensional ball of radius $\delta$  centered at the
origin and orthogonal to $\aff(G)$, then for any $\eta > 0$ we have
\begin{equation}
\label{eq:J1.6.1}
m \{ x  \in Q_\delta : (\1_{M} \ast \nu''_i ) (x) < 1-\eta \} = o( m(Q_\delta) ),
\quad \delta \to 0, \quad i \in \{1,2\}.
\end{equation}

Fix any number $0 < \eta < \half$, and let $D(\delta, \eta)$ denote
the set of all points $x \in Q_\delta$ for which we have
$(\1_M \ast \nu''_i)(x) \geq 1 - \eta$ for both
 $i=1$ and $i=2$. The set $Q_\delta$ has positive
measure, and therefore it follows from
\eqref{eq:J1.6.1} that if $\delta$ is small enough 
then also $D(\delta, \eta)$ has positive measure. On the other hand,
we have
\[
\1_A \ast \mu  \geq \1_A \ast (\mu''_1 + \mu''_2) = 
\1_{A + \tau_1} \ast \nu''_1 + \1_{A + \tau_2} \ast \nu''_2 
\geq \1_M \ast (\nu''_1 +  \nu''_2 ),
\]
where the first inequality  is due to the fact that
the  measures $\mu''_1$ and $\mu''_2$ are obtained by
restricting $\mu$ respectively to the  disjoint sets $S_1$ and $S_2$
defined in \eqref{eq:J2.1.10}, while the second inequality holds since
$M$ is a subset of both $A + \tau_1$ and $A + \tau_2$.
This implies that we have $\1_A \ast \mu \geq 2(1 - \eta) > 1$
a.e.\ on $D(\delta, \eta)$. 
However this contradicts the weak tiling
 assumption $\1_A \ast \mu = \1_{A^\cm}$ a.e.
We conclude that the belt of $A$ generated  by the subfacet $G$ can
have only $4$ or $6$ facets, and this completes the proof
of \thmref{thmJ1.0}. 
\qed


\end{document}